\documentclass[11pt]{amsart}

\usepackage{latexsym}
\usepackage{amsmath}
\usepackage{amsfonts}
\usepackage{amssymb}
\usepackage{hyperref}
\usepackage{mathrsfs}

\usepackage{float}
\usepackage{subcaption}

\usepackage{tikz}
\tikzset{
	element/.style={circle,fill=black,scale=0.5}
}
\usetikzlibrary{calc}
\usepackage[mathscr]{euscript}

\usepackage{enumerate}
\usepackage{cleveref}

\usepackage{color}

\usepackage{xcolor}

\newtheorem{theorem}{Theorem}[section]
\newtheorem{lemma}[theorem]{Lemma}
\newtheorem{corollary}[theorem]{Corollary}

\newtheorem{sublemma}{}[theorem]

\numberwithin{equation}{section} 

\newcommand{\cl}{{\rm cl}}
\newcommand{\si}{{\rm si}}

\newcommand{\delete}{\backslash}

\newcommand{\cC}{\mathcal{C}}
\newcommand{\cS}{\mathcal{S}}

\newcommand{\cW}{\mathcal{W}}

\usepackage[T1]{fontenc}
\usepackage[utf8]{inputenc}

\makeatletter
\g@addto@macro\bfseries{\boldmath} 
\makeatother

\title[The Excluded Minors of Spike Minors]{The Excluded Minors of the Class of Spike Minors}

\author[S.\ Bastida, N.\ Brettell, R.\ Campbell et al.]{Sam Bastida \and Nick Brettell \and Rutger Campbell \and George Drummond \and Emma Hogan \and Charles Semple \and Gerry Toft}

\address{School of Mathematics and Statistics, Victoria University of Wellington, New Zealand}
\email{samuel.a.bastida@gmail.com}

\address{School of Mathematics and Statistics, Victoria University of Wellington, New Zealand}
\email{nick.brettell@vuw.ac.nz}

\address{School of Computing, KAIST, Daejeon, Korea}
\email{rjvdc@rogers.com}

\address{School of Mathematics and Statistics, University of Canterbury, Christchurch, New Zealand}
\email{george.drummond@canterbury.ac.nz}

\address{Mathematical Institute, University of Oxford, Oxford, UK}
\email{emma.hogan@maths.ox.ac.uk}

\address{School of Mathematics and Statistics, University of Canterbury, Christchurch, New Zealand}
\email{charles.semple@canterbury.ac.nz}

\address{School of Mathematics and Statistics, University of Canterbury, Christchurch, New Zealand}
\email{gerrytoft3@gmail.com}



\date{\today}

\begin{document}

\maketitle

\begin{abstract}
Mayhew et al.\ (2021) posed the problem of showing that the minor-closed class of spikes and their minors has a finite set of excluded minors and describing all of them. In this paper, we resolve this problem.
\end{abstract}
	
\section{Introduction}	

Spikes arise in a variety of contexts throughout matroid theory, such as counterexamples to conjectures and matroids with surprising properties. For example, Kahn~\cite{kah88} conjectured that, for each prime power $q$, there is an integer $n(q)$ such that a $3$-connected $GF(q)$-representable matroid has at most $n(q)$ inequivalent $GF(q)$-representations. Oxley et al.~\cite{ox96} showed that the conjecture is false for all $q\ge 7$ using two classes of matroids, one of which is the class of tipped free spikes. Furthermore, Seymour~\cite{sey81} used tipless spikes to show that there is no polynomial-time algorithm to decide if an arbitrary matroid is binary. These examples are indicative of spikes. With justification, Geelen~\cite{gee08} writes ``{\em all of the horrors {\rm (of arbitrary matroids)} are inherent to spikes}''. More recently, subclasses of spikes and their minors have been shown to be ``fractal''.  Let $k$ be a non-negative integer and let $\cS_k$ denote the class of matroids obtained by taking the class of tipless spikes with $k$ circuit-hyperplanes and closing the class under minors. For all even positive integers $2t$, let $\cS_k(2t)$ denote the subset of $2t$-element matroids in $\cS_k$ and let $\mathcal E\mathcal X_k(2t)$ denote the set of $2t$-element excluded minors for $\cS_k$. Mayhew et al.~\cite{may21} showed that, if $k\ge 5$, then the ratio
$$\lim_{t\rightarrow \infty}\frac{|\mathcal E\mathcal X_k(2t)|}{|\cS_k(2t)|+|\mathcal E\mathcal X_k(2t)|}\rightarrow 1.$$
That is, as $t\rightarrow \infty$, the number of $2t$-element excluded minors for $\cS_k$ dominates the number of $2t$-element matroids in $\cS_k$. The class $\cS_k$ is said to be ``weakly fractal''. In the process of proving this result, the authors posed the problem of showing that the class of spikes and their minors have a finite number of excluded minors, and describing these excluded minors~\cite[Problem~3.11]{may21}. In this paper, we resolve this problem and also establish the analogous result for $3$-connected matroids. We next state the main results of the paper.

Let $r\ge 3$ be an integer. A rank-$r$ matroid $M$ with ground set $\{t, x_1, y_1, x_2, y_2, \ldots, x_r, y_r\}$ is a {\em rank-$r$ tipped spike} if $M$ has the following properties:
\begin{enumerate}[(i)]
\item for all $i\in \{1, 2, \ldots, r\}$, the set $\{t, x_i, y_i\}$ is a triangle, and

\item for all non-empty $J \subset \{1, 2, \ldots, r\}$, the set
$$
\bigcup_{i \in J} \{x_i, y_i\}
$$
has rank $|J| + 1$.
\end{enumerate}
We refer to $t$ as the {\em tip} of $M$ and the members of
$$\big\{\{x_i, y_i\}: i\in \{1, 2, \ldots, r\}\big\}$$
as the {\em legs} of $M$. If a matroid can be obtained from $M$ by deleting $t$, it is called a {\em rank-$r$ tipless spike}. In general, a {\em rank-$r$ spike} refers to either a rank-$r$ tipped spike or a rank-$r$ tipless spike. We use $\mathcal S$ to denote the class of spikes and their minors. Note the class of tipless spikes and their minors coincides with the class of tipped spikes and their minors. While the class of spikes is not closed under minors, it is closed under duality~\cite{may21}, and so $\cS$ is also closed under duality.

The first main result of the paper is Theorem~\ref{main1}, which shows that the excluded minors of $\cS$ consist of variants of uniform matroids and their direct sums, as well as three other matroids, which we denote by $P_7^-$, $P_7^=$, and $P_8$. The matroids $P_7^-$ and $P_7^=$ are obtained from the rank-$3$ tipped spike $P_7$ by relaxing exactly one or exactly two circuit-hyperplanes that include the tip, respectively. Geometric representations of $P_7$, $P_7^-$, and $P_7^=$ are given in Figure~\ref{fig: P7 relaxations}, and a geometric representation of $P_8$ is given in Figure~\ref{fig: P8}. Let $\{\ell_1, \ell_2, \ldots, \ell_k\}$ be a multi-set of positive integers exceeding one, and let $\{x_1, x_2, \ldots, x_k\}$ be a subset of the ground set of the uniform matroid $U_{r, n}$. We use $\{\ell_1, \ell_2, \ldots, \ell_k\}$-$U_{r, n}$ to denote the matroid obtained from $U_{r, n}$ by, for all $i\in \{1, 2, \ldots, k\}$, replacing the element $x_i$ with a parallel class of size $\ell_i$. Thus, for example, $\{2, n\}$-$U_{2, 4}$ denotes the matroid obtained from $U_{2, 4}$ by replacing one element with a parallel class of size two and a second element with a parallel class of size $n$. For ease of reading, we will also use $2U_{r, n}$ to denote the matroid obtained from $U_{r, n}$ by replacing each element with a parallel class of size two.

\begin{figure}
\centering
\begin{subfigure}{0.3\textwidth}
\centering
\begin{tikzpicture}[scale=1.5]
\begin{scope}[every node/.style=element]
\node (x1) at (-30:1) {};
\node (x2) at (90:1) {};
\node (x3) at (-150:1) {};
\coordinate (c) at (0,0) {};	
\node (y1) at ($(x1)!0.5!(x2)$) {};
\node (y2) at ($(x2)!0.5!(x3)$) {};
\node (y3) at ($(x3)!0.5!(x1)$) {};
\draw (x1) to (x2) to (x3) to (x1);
\draw (y3) to (x2);
\node (c') at (intersection of  y1--y2 and y3--x2) {};
\draw (y1) to (c') to (y2);
\end{scope}
\end{tikzpicture}
\subcaption{$P_7$} \label{fig: p7}
\end{subfigure}
\begin{subfigure}{0.3\textwidth}
\centering
\begin{tikzpicture} [scale=1.5]
\begin{scope}[every node/.style=element, shift={(3,0)}]
\node (x1) at (-30:1) {};
\node (x2) at (90:1) {};
\coordinate (x3) at (-150:1) {};
\coordinate (c) at (0,0) {};	
\node (y1) at ($(x1)!0.5!(x2)$) {};
\node (y2) at ($(x2)!0.5!(x3)$) {};
\node (y3) at ($(x3)!0.5!(x1)$) {};
\node at ($(x3)!0.2!(x1)$) {};
\draw (x1) to (x2) to (x3) to (x1);
\draw (y3) to (x2);
\node (c') at (intersection of  y1--y2 and y3--x2) {};
\draw (y1) to (c') to (y2);
\end{scope}
\end{tikzpicture}
\subcaption{$P_7^-$} \label{fig: p7-}
\end{subfigure}
\begin{subfigure} {0.3\textwidth}
\centering
\begin{tikzpicture} [scale=1.5]
\begin{scope}[every node/.style=element, shift={(6,0)}]
\coordinate (x1) at (-30:1) {};
\node (x2) at (90:1) {};
\coordinate (x3) at (-150:1) {};
\coordinate (c) at (0,0) {};	
\node (y1) at ($(x1)!0.5!(x2)$) {};
\node (y2) at ($(x2)!0.5!(x3)$) {};
\node (y3) at ($(x3)!0.5!(x1)$) {};
\node at ($(x3)!0.2!(x1)$) {};
\node at ($(x1)!0.2!(x3)$) {};
\draw (x1) to (x2) to (x3) to (x1);
\draw (y3) to (x2);
\node (c') at (intersection of  y1--y2 and y3--x2) {};
\draw (y1) to (c') to (y2);
\end{scope}
\end{tikzpicture}
\subcaption{$P_7^=$}
\end{subfigure}
\caption{The rank-$3$ matroids $P_7$, $P_7^-$, and $P_7^=$.}
\label{fig: P7 relaxations}
\end{figure}
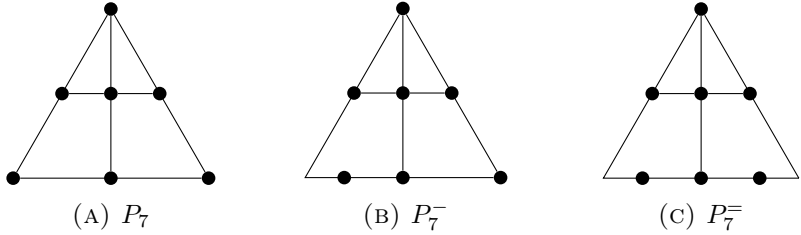

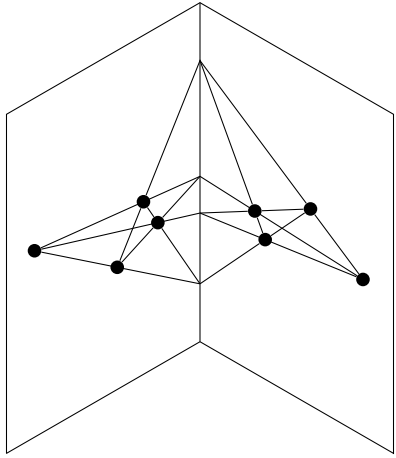
\begin{figure}
\begin{tikzpicture}[scale=1.1]
\begin{scope}[every node/.style=element]
\coordinate (t) at (90:1.8) {};
\coordinate(c) at (0,0);
\coordinate (c1) at (0,0.4);
\coordinate (c3) at (0,-0.9) {};
\coordinate (0010) at (0,-0.5) {};
\node (1) at ($(0010)+(180:2)$) {};
\node (5) at ($(1)!0.5!(c3)$) {};
\node (8) at (intersection of 5--t and 1--c1) {};
\node (4) at (intersection of 5--c1 and 8--c3) {};
\coordinate (c2) at (intersection of 1--4 and c--t) {};
\node (3) at ($(0010)+(-10:2)$) {};
\node (6) at ($(3)!0.6!(c2)$) {};
\node (7) at (intersection of c3--6 and t--3) {};
\node (2) at (intersection of c2--7 and c1--3) {};
\end{scope}
\draw (1) to (c1) to (3) to (c2) to (1) to (c3);
\draw (5) to (8) to (t) to (2) to (6);
\draw (3) to (7) to (t);
\draw (5) to (4) to (c1);
\draw (c2) to (2) to (7);
\draw (8) to (4) to (c3);
\draw (c3) to (6) to (7);
\coordinate (T) at (90:2.5);
\coordinate (B) at (-90:1.6);
\coordinate (TR) at ($(T)+(-30:2.7)$);
\coordinate (BR) at  ($(B)+(-30:2.7)$);
\coordinate (BL) at ($(B)+(210:2.7)$);
\coordinate (TL) at ($(T)+(210:2.7)$);
\draw (BL) to (TL) to (T) to (TR) to (BR) to (B) to (BL);
\draw (T) to (B);
\end{tikzpicture}
\caption{$P_8$, the unique excluded minor for $\cS$ of rank and corank at least four.}
\label{fig: P8}
\end{figure}

\begin{theorem}
A matroid $M$ is a minor of a spike if and only if neither $M$ nor $M^*$ has a minor isomorphic to one of the following matroids:
\begin{enumerate}[{\rm (i)}]
\item $U_{1, 5}\oplus U_{0, 2}$,

\item $U_{2, 4}\oplus U_{0, 1}$,\, $U_{1, 3}\oplus U_{1, 1}\oplus U_{0, 1}$,\, $U_{2, 6}$,\, $U_{1, 3}\oplus U_{1, 3}$,\, $\{2\}$-$U_{2, 3}\oplus U_{1, 2}$,\, $\{2, 2, 2\}$-$U_{2, 4}$,\, $\{2, 2\}$-$U_{2, 5}$,

\item $P^-_7$,\, $P^=_7$, and

\item $P_8$.
\end{enumerate}
\label{main1}
\end{theorem}

\noindent The matroids in Theorem~\ref{main1} are grouped so that the matroids in (i), (ii), (iii), and (iv) each have rank one, two, three, and four, respectively.
	
In addition to Theorem~\ref{main1}, we prove the analogue of this theorem for when the matroids under consideration are $3$-connected. Let $\cS_3$ denote the subset of $\cS$ consisting of the $3$-connected members of $\cS$. Note that, while $\cS_3$ is not closed under minors, it is closed under duality. Let $O_7$ be the matroid obtained from the rank-$3$ wheel by freely adding an element to a rank-$2$ flat of size $3$. Similarly, let $O_7^-$ be the matroid obtained from the rank-$3$ whirl by freely adding an element to a rank-$2$ flat of size $3$. Geometric representations of $O_7$ and $O_7^-$ are given in Figure~\ref{L7 and L7'}. Observe that $O_7^-$ is obtained from $O_7$ by relaxing any one of its three circuit-hyperplanes. We denote by $AG(2, 3)\backslash e$ the unique (up to isomorphism) single-element deletion of the rank-$3$ ternary affine geometry. The second main result of this paper is Theorem~\ref{them: 3con}.

\begin{figure}
\centering
\begin{subfigure}{0.3\textwidth}
\centering
\begin{tikzpicture}[scale=1.5]
\begin{scope}[every node/.style=element]

\node (x1) at (-30:1) {};
\node (x2) at (90:1) {};
\node (x3) at (-150:1) {};
\coordinate (c) at (0,0) {};
\node (y1) at ($(x1)!0.5!(x2)$) {};
\node (y2) at ($(x2)!0.5!(x3)$) {};
\node (y3) at ($(x3)!0.5!(x1)$) {};
\node at ($(x1)!0.5!(y1)$) {};
\draw (x1) to (x2) to (x3) to (x1);
\draw (0,0) circle (0.5);
		
\end{scope}
\end{tikzpicture}
\subcaption{$O_7$}
\end{subfigure}
\begin{subfigure} {0.3\textwidth}
\centering
\begin{tikzpicture}[scale=1.5]
\begin{scope}[every node/.style=element, shift={(3,0)}]
\node (x1) at (-30:1) {};
\node (x2) at (90:1) {};
\node (x3) at (-150:1) {};
\coordinate (c) at (0,0) {};
\node (y1) at ($(x1)!0.5!(x2)$) {};
\node (y2) at ($(x2)!0.5!(x3)$) {};
\node (y3) at ($(x3)!0.5!(x1)$) {};
\node at ($(x1)!0.5!(y1)$) {};
\draw (x1) to (x2) to (x3) to (x1);
\end{scope}
\end{tikzpicture}
\subcaption{$O_7^-$}
\end{subfigure}
\caption{The matroids $O_7$ and $O_7^-$.}
\label{L7 and L7'}
\end{figure}
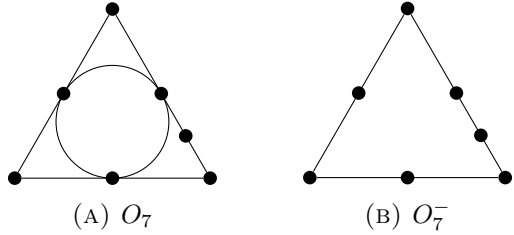

\begin{theorem}
\label{them: 3con}
A $3$-connected matroid $M$ is a minor of a spike if and only if neither $M$ nor $M^*$ has a minor isomorphic to one of $U_{2,6}$,\, $O_7$,\, $O_7^{-}$,\, $P_7^-$,\, $P_7^=$,\, $AG(2, 3)\backslash e$,\, $P_8$,\, $M(\cW_4)$, and $\cW^4$.
\end{theorem}

%
%

As the self-dual matroid $M(\cW_4)$ is the unique binary matroid in the list of excluded minors in Theorem~\ref{them: 3con}, we immediately get the following result of Oxley~\cite{ox87} as a corollary.

\begin{corollary}
A $3$-connected binary matroid $M$ is a minor of a spike if and only if $M$ has no minor isomorphic to $M(\cW_4)$.
\label{binary}
\end{corollary}

Similarly, we immediately establish the ternary analogue of Corollary~\ref{binary}.

\begin{corollary}
A $3$-connected ternary matroid $M$ is a minor of a spike if and only if $M$ has no minor isomorphic to one of $M(\cW_4)$,\, $\cW^4$,\, $O_7$,\, $O_7^*$,\, $AG(2,3)\backslash e$,\, $(AG(2,3)\backslash e)^*$, or $P_8$.
\end{corollary}

The paper is organised as follows. Notation and terminology follows Oxley~\cite{ox11}, with the exception that our definition of the ``leg'' of a spike does not include the tip. In the next section, we detail some lemmas that are used throughout the paper. The proof of Theorem~\ref{main1} is given in Section~\ref{proof1}, while the proof of Theorem~\ref{them: 3con} is given in Section~\ref{proof2}.

\section{Preliminaries}

In this section, we state some lemmas that will be used in the paper. The first two lemmas characterise spikes in terms of circuits (see, for example,~\cite{ox11}). To ease reading, a {\em $4$-circuit} and a {\em $4$-cocircuit} is a $4$-element circuit and a $4$-element cocircuit, respectively.

\begin{lemma}
Let $r\ge 3$, and let $M$ be a matroid with ground set $\{t, x_1, y_1, x_2, y_2, \ldots, x_r, y_r\}$. Then $M$ is a rank-$r$ tipped spike with legs $L_1=\{x_1, y_1\}, L_2=\{x_2, y_2\}, \ldots, L_r=\{x_r, y_r\}$ if and only if the set of non-spanning circuits of $M$ is $\cC_1 \cup \cC_2 \cup \cC_3$, where
\begin{enumerate}[{\rm (i)}]
\item $\cC_1 = \{L_i\cup \{t\}: 1\le i\le r\}$

\item $\cC_2 = \{L_i\cup L_j: 1\le i < j\le r\}$

\item $\cC_3$ is a, possibly empty, subset of
$$\big\{\{z_1, z_2, \ldots, z_r\}: \mbox{$z_i\in \{x_i, y_i\}$ for all $i\in \{1, 2, \ldots, r\}$}\big\},$$
such that no two members of $\cC_3$ have more than $r-2$ elements in common.
\end{enumerate}
\label{circuits}
\end{lemma}

\noindent The circuits given in (iii) of Lemma~\ref{circuits} are called {\em traversals}. A spike with no traversals is said to be {\em free}.

\begin{lemma}
Let $r\ge 3$ and let $M$ be a matroid on $2r$ elements. Then $M$ is a rank-$r$ spike if and only if the ground set of $M$ can be partitioned into pairs $\{L_1, L_2, \ldots, L_r\}$ such that, for all distinct $i, j\in \{1, 2, \ldots, r\}$, the union of $L_i$ and $L_j$ is a $4$-circuit and a $4$-cocircuit.
\label{pairs}
\end{lemma}

The next lemma characterises tipped spike minors as restrictions of particular matroids, and is used frequently in the proof of Theorem~\ref{main1}. This lemma is separated into two cases. The first of these is concerned with tipped spike minors in which the tip has been contracted, and the second with tipped spike minors in which the tip has not been contracted. Its straightforward proof is omitted.

\begin{lemma}
\label{lemma: spike degeneracy}
Let $r\ge 3$, and let $M$ be a rank-$r$ matroid. Then $M$ is a minor of a spike if and only if, for some positive integer $n$, the matroid $M$ is isomorphic to a restriction of either
\begin{enumerate}[{\rm (i)}]
\item $2U_{r, r+1}\oplus U_{0, n}$, or

\item a rank-$r$ tipped spike in which the tip has been replaced by a parallel class of size $n$.
\end{enumerate}
\end{lemma}

\noindent We refer to the matroid in Lemma~\ref{lemma: spike degeneracy}(ii) as an {\em overloaded tipped spike}. If $M$ is a restriction of a spike and $L$ is a leg of that spike, we shall say that the set $L'$ of elements of $L$ which are still elements of $M$ is a leg of $M$. If $|L'| = 2$, then we say that $L'$ is a {\em full} leg, and if $|L'| = 1$, we say that $L'$ is a {\em half} leg.

Let $e$ be an element in a half leg of a spike minor $M$. The union of $e$ and each full leg of $M$ is a triad. Thus, $M^*$ is a spike minor with $e$ at the tip. We will sometimes refer to $e$ is a {\em cotip} of $M$. In particular, if $M$ is a tipped spike minor with a single non-tip element deleted, so that $M$ has exactly one half leg, we say that $M$ is a {\em spike with tip and cotip}. The following characterisation of spikes with tip and cotip is a straightforward consequence of Lemma~\ref{circuits}.

\begin{lemma}
A matroid $M$ is a spike with tip and cotip if and only if there exist distinct $t, t^* \in E(M)$ such that $E(M) - \{t,t^*\}$ can be partitioned into pairs $\{x_1,y_1\},\{x_2,y_2\},\ldots,\{x_r,y_r\}$, where $\{x_i,y_i,t\}$ is a triangle and $\{x_i,y_i,t^*\}$ is a triad for all $i \in \{1,2,\ldots,r\}$.
\label{lem: tip_and_cotip}
\end{lemma}

\begin{proof}
Let $i$ and $j$ be distinct elements of $\{1,2,\ldots,r\}$. First, we show that $\{x_i,y_i,x_j,y_j\}$ is a circuit. Circuit elimination between $\{t,x_i,y_i\}$ and $\{t,x_j,y_j\}$ implies that $M$ has a circuit contained in $\{x_i,y_i,x_j,y_j\}$. By orthogonality with the triads $\{t^*,x_i,y_i\}$ and $\{t^*,x_j,y_j\}$, the set $\{x_i,y_i,x_j,y_j\}$ is a circuit.

Let $C$ be a circuit of $M$ which is not of the form $\{t,x_i,y_i\}$ or $\{x_i,y_i,x_j,y_j\}$. Orthogonality with the triads implies that $C$ contains $t^*$ and an element of $\{x_i,y_i\}$ for all $i \in \{1,2,\ldots,r\}$. It follows from Lemma~\ref{circuits} that the circuits of $M$ are the circuits of a deletion of a non-tip element from a tipped spike. Thus, $M$ is a spike with tip and cotip.
\end{proof}

One of the consequences of Theorem~\ref{main1} is that $P_8$ is the only excluded minor for $\cS$ with rank and corank at least $4$. The following characterisation of $P_8$ (see, for example, \cite[p.\ 651]{ox11}) will be used to prove this.

\begin{lemma} 
Let $M$ be a matroid. Then $M \cong P_8$ if and only if $M \backslash e \cong P_7$ and $M / e \cong P_7^*$ for all $e \in E(M)$.
\label{lem: p8}
\end{lemma}

The last four lemmas of this section are used in the proof of Theorem~\ref{them: 3con}. The first and second are well known and easily proved, the third is due to Oxley~\cite{ox89}, and the fourth is Tutte's Triangle Lemma~\cite{tut66}.

\begin{lemma}
\label{3conspike}
Let $\Phi_r$ be a tipped spike with tip $t$ and rank~$r$, where $r\ge 4$. Then the $3$-connected rank-$r$ restrictions of $\Phi_r$ are $\Phi_r\backslash t$, $\Phi_r\backslash e$, and $\Phi_r\backslash \{t, e\}$ for all $e\in E(\Phi_r)-\{t\}$.
\end{lemma}

\begin{lemma}
\label{segdelete}
Let $M$ be a $3$-connected matroid and let $X\subseteq E(M)$ such that $M|X\cong U_{2, n}$, where $n\geq 4$. Then $M\delete x$ is $3$-connected for all $x\in X$.
\end{lemma}

\begin{lemma}
\label{lem: U25 U35}
Let $M$ be a $3$-connected matroid with rank and corank at least three. Then $M$ has a $U_{2, 5}$-minor if and only if $M$ has a $U_{3, 5}$-minor.
\end{lemma}

\begin{lemma}
\label{lem: tt}
Let $M$ be a $3$-connected matroid having at least four elements, and suppose that $\{e, f, g\}$ is a triangle of $M$ such that neither $M\backslash e$ nor $M\backslash f$ is $3$-connected. Then $M$ has a triad that contains $e$ and exactly one of $f$ and~$g$.
\end{lemma}

\section{Proof of Theorem~\ref{main1}}
\label{proof1}

In this section, we prove Theorem~\ref{main1}. We begin by establishing the minors of spikes of rank at most two.

\begin{lemma}
\label{lem:smaller_spike_minors}
Let $r\in \{1, 2\}$ and let $M$ be a rank-$r$ matroid. Then $M$ is a minor of a spike if and only if, for some positive integer $n$, the matroid $M$ is isomorphic to a restriction of one of the matroids
\begin{enumerate}[{\rm (i)}]
\item $2U_{2, 3}\oplus U_{0, n}$, $\{n\}$-$U_{2,5}$, $\{2, n\}$-$U_{2,4}$, $\{2, 2, n\}$-$U_{2,3}$,

\item $U_{1, 4}\oplus U_{0, n}$, and $U_{1, n}\oplus U_{0, 1}$.
\end{enumerate}
\end{lemma}

\begin{proof}
Let $M$ be a rank-$2$ spike minor. Then $M$ is a contraction of a rank-$3$ spike minor and so, by Lemma~\ref{lemma: spike degeneracy}, the matroid $M$ is either a minor of $2U_{3, 4} \oplus U_{0,m}$ for some positive integer $m$, or a minor of a rank-$3$ tipped spike $N$ in which the tip has been replaced by a parallel class of size $n$ for some positive integer $n$. Each rank-$2$ minor of $2U_{3, 4} \oplus U_{0, m}$ is a restriction of $2U_{2, 3} \oplus U_{0, m+1}$. Thus, if the former holds, then $M$ is a restriction of $2U_{2,3} \oplus U_{0,m+1}$. Otherwise, $M$ is a restriction of $N / e$ for some element $e$ of $N$. If $e$ is in the parallel class of size $n$, then $N/e$ is isomorphic to $2U_{2, 3}\oplus U_{0, n-1}$. Otherwise, $N/e$ consists of the parallel class of size $n$ and four additional elements. Depending on whether $e$ is in zero, one, or two traversals, $N/e$ is isomorphic to $\{n+1\}$-$U_{2, 5}$, $\{2, n+1\}$-$U_{2, 4}$, or $\{2, 2, n+1\}$-$U_{2,3}$, respectively. In each case, $M$ is a restriction of one of the matroids in (i). Conversely, each of the matroids in (i) is a rank-$2$ minor of a spike, which completes the proof of the lemma when $M$ has rank~$2$.

Now let $M$ be a rank-$1$ spike minor. Then $M$ is a contraction of a rank-$2$ spike minor, so $M$ is a restriction of a matroid obtained by contracting a non-loop element from one of the matroids in (i). Every matroid obtained in this way is isomorphic to either $U_{1,4} \oplus U_{0,n}$ or $U_{1,n} \oplus U_{0,1}$, for some positive integer $n$, so $M$ is a restriction of one of the matroids in (ii). Each of the matroids in (ii) is a rank-$1$ minor of a spike, which completes the proof of the lemma.
\end{proof}

The rest of the proof of Theorem~\ref{main1} essentially consists of the next two lemmas. The first lemma determines the excluded minors of $\cS$ of rank at most three and the second lemma determines the excluded minors of $S$ of rank and corank at least four. 

\begin{lemma}
\label{lemma: small spike-minors}
Let $M$ be a matroid with rank at most three. Then $M$ is an excluded minor for $\cS$ if and only if $M$ is isomorphic to one of the matroids
\begin{enumerate}[{\rm (i)}]
\item $U_{1, 5}\oplus U_{0, 2}$,

\item $U_{2, 4}\oplus U_{0, 1}$,\, $U_{2, 6}$,\, $U_{1, 3}\oplus U_{1, 3}$,\, $\{2, 2, 2\}$-$U_{2, 4}$,\, $\{2, 2\}$-$U_{2, 5}$,\, $U_{1, 3}\oplus U_{1, 1}\oplus U_{0, 1}$,

\item $U_{2, 4}\oplus U_{1, 1}$,\, $U_{2, 3}\oplus U_{1, 1}\oplus U_{0, 1}$,\, $\{2\}$-$U_{2, 3}\oplus U_{1, 2}$,\, $P_7^-$, and $P_7^=$.
\end{enumerate}
\end{lemma}

\begin{proof}
Using Lemmas~\ref{lemma: spike degeneracy} and~\ref{lem:smaller_spike_minors}, it is straightforward to verify that each of the matroids in (i), (ii), and (iii) is an excluded minor for $\cS$. It remains to show that there are no further excluded minors for $\cS$ of rank at most three. We check each rank in turn, noting that every rank-$0$ matroid is a spike minor.

\begin{sublemma}
The matroid $U_{1,5} \oplus U_{0,2}$ is the only rank-$1$ excluded minor for $\cS$.
\label{rank1}
\end{sublemma}

Let $M$ be a rank-$1$ excluded minor for $\cS$. Then $M\cong U_{1, n}\oplus U_{0, m}$ for some positive integers $m$ and $n$. If $n\geq 5$ and $m\geq 2$, then $M$ has a minor isomorphic to $U_{1,5} \oplus U_{0,2}$. If either $n < 5$ or $m < 2$, then $M$ is a spike minor by Lemma~\ref{lem:smaller_spike_minors}(ii). Thus (\ref{rank1}) holds.

\begin{sublemma}
The matroids listed in (ii) are the only rank-$2$ excluded minors for $\cS$.
\label{rank2}
\end{sublemma}

Let $M$ be an additional rank-$2$ excluded minor for $\cS$. If $M$ has a loop, then $M$ contains at most three (including trivial) distinct parallel classes; otherwise, $M$ has a minor isomorphic to $U_{2,4} \oplus U_{0,1}$. If every parallel class of $M$ has size at most two, then $M$ is a restriction of $2U_{2, 3}\oplus U_{0, n}$ and so, by Lemma~\ref{lem:smaller_spike_minors}, $M$ is in $\cS$. Therefore, $M$ has a parallel class of size at least three, and so $M$ has a minor isomorphic to $U_{1, 3}\oplus U_{1, 1}\oplus U_{0,1}$, another contradiction. Hence, $M$ has no loops.

Let $\si(M)\cong U_{2, n}$ for some positive integer $n\ge 2$. Since $M$ has no minor isomorphic to $U_{2, 6}$, we have that $n < 6$. Furthermore, as $M$ has no minor isomorphic to $U_{1,3} \oplus U_{1,3}$, $M$ has at most one parallel class containing at least three elements. Hence, if $n \leq 3$, Lemma~\ref{lem:smaller_spike_minors} implies that $M$ is in $\cS$. Therefore $n\in \{4, 5\}$. If $n = 4$, then, as $M$ does not have a minor isomorphic to $\{2,2,2\}$-$U_{2, 4}$, we have that $M$ has at most two non-trivial parallel classes. But then, by Lemma~\ref{lem:smaller_spike_minors}, $M$ is a spike minor, a contradiction. If $n = 5$, then, as $M$ does not have a minor isomorphic to $\{2, 2\}$-$U_{2, 5}$, it follows that $M$ has at most one non-trivial parallel class. But, again by Lemma~\ref{lem:smaller_spike_minors}, $M$ is in $\cS$. This completes the proof of (\ref{rank2}).

\begin{sublemma}
The matroids listed in (iii) are the only rank-$3$ excluded minors for~$\cS$.
\label{rank3}
\end{sublemma}

Let $M$ be an additional rank-$3$ excluded minor for $\cS$. If $M$ has a $U_{2, 4}$-restriction, then $M$ has a minor isomorphic to $U_{2, 4}\oplus U_{1, 1}$, and so $M$ has no such restriction. Suppose $M$ contains neither a triangle nor a parallel class of size at least three. As $r(M) = 3$ and $M$ contains no triangles, we have that $\si(M)\cong U_{3, m}$ for some $m\geq 3$. Since $M$ is not a spike minor, Lemma~\ref{lemma: spike degeneracy} implies that $m\geq 5$. If $M$ has a loop, then $M$ has a minor isomorphic to $U_{3, 5}\oplus U_{0,1}$, and so $M$ has a minor isomorphic to $U_{2, 4}\oplus U_{0, 1}$, a contradiction. Furthermore, if $M$ has a pair of elements in parallel, then contracting one of these elements implies that $M$ has a minor isomorphic to $U_{2, 4}\oplus U_{0, 1}$. So $M$ has no loops or parallel elements. Since $U_{3, 5}$ and $U_{3, 6}$ are both spike minors, $m\geq 7$. But now $M$ has a $U_{2, 6}$ minor, a contradiction. Hence $M$ has either a triangle or a parallel class of size at least three.

Suppose that $M$ has a triangle $T$. If $M$ has a loop, then $M$ has a restriction isomorphic to $U_{2, 3}\oplus U_{1, 1}\oplus U_{0, 1}$. Thus $M$ has no loops. Assume that $M$ has two distinct pairs of parallel elements $\{a_1, a_2\}$ and $\{b_1, b_2\}$ such that $r(\{a_1, a_2, b_1, b_2\}) = 2$. If $\{a_1, a_2, b_1, b_2\}\subseteq \cl(T)$, then $M/a_1$ has a minor isomorphic to $U_{1, 3}\oplus U_{1, 1}\oplus U_{0, 1}$, a contradiction. If $\{a_1, a_2\}\subseteq \cl(T)$ and $\{b_1, b_2\}\not\subseteq \cl(T)$, then $M$ has a minor isomorphic to $\{2\}$-$U_{2, 3}\oplus U_{1, 2}$. Thus we may assume that no element in $\{a_1, a_2, b_1, b_2\}$ is contained in a triangle. But now $M/ a_1$ has a minor isomorphic to $U_{2, 4}\oplus U_{0, 1}$. Hence, $M$ has at most one non-trivial parallel class.

Assume that $M$ has a pair of parallel elements $\{a_1, a_2\}$. Suppose that $\{a_1, a_2\}\not\subseteq \cl(T)$. By Lemma~\ref{lemma: spike degeneracy}(ii), the matroid $U_{2, 3}\oplus U_{1, n}$ is a spike minor for all positive integers $n$. Therefore, there exists an element $e\in E(M) - (T\cup \{a_1, a_2\})$ such that $r(\{a_1, e\})=2$. The matroid $M / a_1$ does not have a minor isomorphic to $U_{2,4} \oplus U_{0,1}$, which implies that there is a triangle $T'$ of $M$ which contains $a_1$, $e$, and an element of $T$. Furthermore, this is true for all elements of $E(M) - (T \cup \{a_1,a_2\})$ which are not in parallel with $a_1$ and $a_2$. Since $M$ has no $U_{2,4}$-restriction, and each of these elements are in a triangle with $a_1$ and an element of $T$, it follows that there are at most three such elements. But this implies that $M$ is a spike minor with the parallel pair $\{a_1,a_2\}$ at the tip and a traversal $T$. This is a contradiction, so $\{a_1,a_2\}$ is in the closure of every triangle of $T$.

Choose a maximal set $\{x_1,x_2,\ldots,x_k\}$ of elements of $E(M) - \{a_1\}$ such that $\{a_1,x_1,x_2,\ldots,x_k\}$ does not contain a non-spanning circuit. Each $x_i$ is either not contained in a triangle, or is contained in a triangle with $a_1$ and an element outside of $\{a_1,x_1,x_2,\ldots,x_k\}$. In particular, if $k \leq 3$, then $M$ is a spike minor with the parallel pair $\{a_1,a_2\}$ at the tip, and each $x_i$ an element from a separate leg. Thus, $k \geq 4$. But now $\{x_1,x_2,x_3,x_4\}$ is a $U_{2,4}$-restriction in $M / a_1$, so $M$ has a minor isomorphic to $U_{2,4} \oplus U_{0,1}$. Hence, $M$ has no parallel elements.

Let $e\in E(M)$. By Lemma~\ref{lemma: spike degeneracy}, the spike minor $E(M \backslash e)$ has at most $7$ elements. If $|E(M)|\le 6$, then it is easily checked that $M$ is a spike minor. Thus $|E(M)|\in \{7, 8\}$. Say $|E(M)| = 8$. Then $e$ is contained in at most three triangles. If $e$ is contained in zero or one triangles, then $M/e$ has a restriction isomorphic to $U_{2, 6}$. If $e$ is contained in exactly two triangles, then $M/e$ is isomorphic to $\{2, 2\}$-$U_{2, 5}$, while if $e$ is contained in exactly three triangles, then $M/e$ is isomorphic to $\{2,2,2\}$-$U_{2, 4}$. Therefore, $|E(M)| = 7$. If $e$ is not contained in a triangle, then $M/e\cong U_{2, 6}$. Hence, every element of $M$ is contained in a triangle. If $M$ has an element common to three triangles, then $M$ is a spike minor. Now, $M$ must have two disjoint triangles, $T$ and $T'$, say. The unique element $f$ of $E(M) - (T \cup T')$ is contained in at least one triangle, and cannot be contained in three triangles. If $f$ is contained in one triangle, then $M \cong P_7^=$, and if $f$ is contained in two triangles, then $M \cong P_7^-$. This is a contradiction, so $M$ has no triangles.

The last case to consider is when $M$ has no triangles, but has a parallel class $\{a_1, a_2, \ldots, a_k\}$, where $k\ge 3$. If $M$ has a second parallel class $\{b_1,b_2,\ldots,b_\ell\}$ with $\ell \geq 2$, then $M / b_1$ has a minor isomorphic to $U_{1, 3}\oplus U_{1, 1}\oplus U_{0, 1}$, a contradiction. Therefore, $\{a_1, a_2, \ldots, a_k\}$ is the only parallel class of $M$ of size at least two. If $|E(M) - \{a_1, a_2, \ldots, a_k\}|\leq 3$, then $M$ is a spike minor. If $|E(M) - \{a_1, a_2, \ldots, a_k\}| > 3$, then, as $M$ has no triangles, $M/a_1$ has a minor isomorphic to $U_{2, 4}\oplus U_{0, 1}$. This last contradiction proves (\ref{rank3}).

Combining (\ref{rank1})--(\ref{rank3}) completes the proof of the lemma.
\end{proof}

The next lemma shows that there is a unique excluded minor for $\mathcal{S}$ of rank and corank at least four, and in doing so, complete the proof of Theorem~\ref{main1}. The proof of Lemma~\ref{rank4} is a lengthy case analysis. It begins by showing that if $M$ is an excluded minor for $\mathcal S$ and $r(M), r^*(M)\ge 4$, then $M$ has certain properties and relatively quickly deduces that $r(M)=4=r^*(M)$. Most of the work in proving the lemma is to handle the case $r(M)=4=r^*(M)$. The primary reason for this is that a $4$-circuit of a rank-$4$ spike need not be a union of legs; it could instead be a traversal.  

\begin{lemma}
The matroid $P_8$ is the unique excluded minor for $\cS$ of rank and corank at least four.
\label{rank4}
\end{lemma}

\begin{proof}
Since $P_7$ is a rank-$3$ tipped spike, it follows from Lemma~\ref{lem: p8} that $P_8$ is an excluded minor for $\cS$. Now, let $M$ be an arbitrary excluded minor for $\cS$ of rank and corank at least four. Since $\cS$ is closed under duality, $M^*$ is also such a matroid. For convenience, let $r=r(M)$ and $r^*=r^*(M)$. The proof of (\ref{simple}) is given in~\cite[Lemma 3.9]{may21}. However, for the sake of completeness, we include a proof here.

\begin{sublemma}
$M$ is simple and cosimple.
\label{simple}
\end{sublemma}

Suppose that $M$ has a loop $\ell$. If $M$ has at least two loops, then, by Lemma~\ref{lemma: spike degeneracy}, $M\backslash \ell$ is a restriction of $2U_{r, r+1}\oplus U_{0, n}$ for some positive integer $n$, and so $M\in \cS$, a contradiction. So $M$ has exactly one loop. By Lemma~\ref{lemma: spike degeneracy}, for all $e\in E(M)-\{\ell\}$, the matroid $M\backslash e$ is a restriction of $2U_{r, r+1}\oplus U_{0, 1}$. Thus, all parallel classes of $M$ have size at most two. Moreover, $M$ has at least one non-trivial parallel class $\{f,g\}$, as otherwise $r^* \leq 3$. Now, $M\backslash f$ is a restriction of $2U_{r, r+1}\oplus U_{0, 1}$, so $M$ is also such a restriction. This contradiction implies that $M$ has no loops.

Now suppose that $M$ has a parallel class $P$ of size at least two. Say $M$ has another non-trivial parallel class $P'$. Let $e\in E(M)-(P\cup P')$. By Lemma~\ref{lemma: spike degeneracy}, $M\backslash e$ is a restriction of $2U_{r, r+1}$. Thus $|P|=2$ and $|P'|=2$. In turn, this implies that all non-trivial parallel classes of $M$ have size two. If $e$ is in a non-trivial parallel class of $M$, then, since $M \backslash e$ is a restriction of $2U_{r,r+1}$, we have that $M$ is also a restriction of $2U_{r,r+1}$, a spike minor. It now follows that $P$ and $P'$ are the only non-trivial parallel classes of $M$. Since $r^*(M)\ge 4$, we have that $M\backslash e \cong \{2, 2\}$-$U_{r, r+1}$. Let $p' \in P'$. Since $M \backslash p'$ has only one non-trivial parallel class, Lemma~\ref{lemma: spike degeneracy} implies that $M \backslash p'$ is a restriction of an overloaded tipped spike. Since $r^* \geq 4$, there is a full leg $\{f,f'\}$ of $M \backslash p'$. Now, $r(P \cup \{f,f'\}) = 2$, so there exists $g \notin \cl(P \cup P' \cup \{f,f'\})$. The matroid $M \backslash g$ contains a triangle and two distinct non-trivial parallel classes. This contradicts Lemma~\ref{lemma: spike degeneracy}. Hence, $P$ is the only non-trivial parallel class of $M$.

Let $p\in P$. Since $r^* \geq 4$ and $M \backslash p$ has at most one non-trivial parallel class, the matroid $M \backslash p$ is not a restriction of $2U_{r,r+1}$. Thus, $M \backslash p$ is a restriction of an overloaded tipped spike. If $|P| \geq 3$, then $P-\{p\}$ is at the tip of $M \backslash p$. But this implies that $M$ is also a restriction of an overloaded tipped spike. So $|P| = 2$.

Let $P=\{p, q\}$ and let $e\in E(M)-P$. Since $r^* \geq 4$, the matroid $M\backslash e$ is a restriction of a tipped spike with $P$ at the tip and at least one full leg. Thus, there is a triangle $T$ of $M$ containing $p$. Similarly, let $f\in T - P$. Then $M\backslash f$ is a restriction of a tipped spike with $P$ at the tip and a triangle $T' \neq T$ containing $p$. In $M \backslash q$, the element $p$ is contained in the intersection of two triangles. Furthermore $r(M \backslash q) \geq 4$, so $M \backslash q$ is a tipped spike with tip $p$. But now $M$ is also a spike minor. Thus, $M$ is simple, and, by duality, $M$ is also cosimple.

\begin{sublemma}
\label{lem:no_triangles_or_triads_spikes}
$M$ has no triangles and no triads.
\end{sublemma}

Suppose that $M$ has a triangle, and let $e \in E(M)$ be an element outside of this triangle. Since $M \backslash e$ is simple, Lemma~\ref{lemma: spike degeneracy} implies that $M \backslash e$ is a restriction of a tipped spike. Moreover, $M \backslash e$ has a triangle, and has rank at least $4$, so $M \backslash e$ has a tip $t$. Now, $r^*(M\backslash e) \geq 3$, so $M\backslash e$ has at least two full legs. Thus, $t$ is contained in at least two triangles of $M$.

Consider $M/t$. Since $M/t$ contains at least two parallel pairs, Lemma~\ref{lemma: spike degeneracy} implies that $M / t$ is a restriction of $2U_{r-1,r}$. If $M / t$ is a restriction of $2U_{r-1,r-1}$, then $M / t$, and thus also $M$, contains series pairs. This contradiction implies that $\si(M / t) \cong U_{r-1,r}$. Since $M$ is cosimple, there is at most one parallel class of $M/t$ consisting of a single element, and all other parallel classes of $M/t$ have size two. 

Suppose there is a parallel class of $M / t$ consisting of a single element $t^*$, and let $\{x_1,y_1\},\{x_2,y_2\},\ldots,\{x_{r-1},y_{r-1}\}$ be the parallel pairs of $M/t$. For all $i \in \{1,2,\ldots,r-1\}$, the set $\{x_i,y_i,t^*\}$ is a triad of $M / t$, so also of $M$. Also, $\{x_i,y_i,t\}$ is a triangle of $M$, so, by Lemma~\ref{lem: tip_and_cotip}, $M$ is a spike with tip and cotip. This is a contradiction.

Otherwise, every parallel class of $M / t$ has two elements. Let $\{x_1,y_1\},\{x_2,y_2\},\ldots,\{x_r,y_r\}$ be these parallel classes. For all distinct $i,j \in \{1,2,\ldots,r\}$, the set $\{x_1,y_1,x_2,y_2\}$ is a cocircuit of $M / t$, so also a cocircuit of $M$. Since $r(M)\geq 4$, there exists some element $f\not\in \{x_i,y_i,x_j,y_j\}$. Since $\{t,x_i,y_i\}$ and $\{t,x_j,y_j\}$ are triangles of $M\backslash f$ and $r(M\backslash f)\ge 4$, the matroid $M\backslash f$ is a restriction of a rank-$r$  tipped spike with legs $\{x_i,y_i\}$ and $\{x_j,y_j\}$. Hence, $\{x_i,y_i,x_j,y_j\}$ is also a circuit of $M$. Thus, $M$ is a tipped spike. This contradiction implies that $M$ has no triangles and, by duality, no triads, so (\ref{lem:no_triangles_or_triads_spikes}) holds.

\begin{sublemma}
\label{lem:no_u35_spikes}
Neither $M$ nor $M^*$ has a $U_{3,5}$-restriction.
\end{sublemma}

Suppose that $M$ has a $U_{3,5}$-restriction. Then, since $r(M)\geq 4$, there is some element $x$ in $E(M)$ such that $M\backslash x$ also has a $U_{3,5}$-restriction. But, by Lemma \ref{lemma: spike degeneracy}, $M\backslash x$ has no such restriction as $r(M)\ge 4$. Therefore $M$ and, by duality, $M^*$ have no $U_{3, 5}$-restriction. This proves (\ref{lem:no_u35_spikes}).

Since both $M$ and $M^*$ are simple and have no triangles, (\ref{lem:no_u35_spikes}) is equivalent to saying that no $5$-element subset of $E(M)$ has rank or corank three.

\begin{sublemma}
\label{rank bound}
$\min\{r(M), r^*(M)\}\leq 4$. 
\end{sublemma}

Without loss of generality, suppose that $r(M)\geq r(M^*)\geq 5$, and let $z\in E(M)$. Since $M$ is simple and has no triangles, and $r(M^*)\ge 5$, Lemma~\ref{lemma: spike degeneracy} implies that $M\backslash z$ is a restriction of a rank-$r$ tipless spike. Since $r^*(M\backslash z)\geq 4$, it follows that $M\backslash z$ has at least $k$ full legs $L_1, L_2, \dots, L_{k}$, where $k=r^*(M)-1\ge 4$. Let $L_k=\{x_k, y_k\}$.

If $M\backslash z$ has two half legs $L$ and $L'$, then $L\cup L'$ is a series pair of $M\backslash z$, and so either $L\cup L'$ is a series pair of $M$ or $L\cup L'\cup \{z\}$ is a triad of $M$, a contradiction. Thus $M\backslash z$ has at most one half leg, and so $k\in \{r-1, r\}$.

Consider $M\backslash y_k$. Again by Lemma~\ref{lemma: spike degeneracy}, $M\backslash y_k$ is a rank-$r$ tipless spike. Since $r(M\backslash y_k)\ge 5$, a traversal of $M\backslash y_k$ has at least five elements, and so all $4$-circuits of $M\backslash y_k$ are unions of pairs of legs. Let $i \in \{1,2,\ldots,r-1\}$. Since $k-1 \geq 3$, there exist distinct $j, k \in \{1,2,\ldots,r-1\} - \{i\}$. Both $L_i \cup L_j$ and $L_i \cup L_j$ are $4$-element circuits of $M \backslash y_k$, so each is a union of two legs. The only possibility is that $L_i$ is a leg of $M \backslash y_k$. More generally, each of $L_1, L_2, \ldots, L_{k-1}$ is a leg of $M \backslash y_k$.

If $\{x_k, z\}$ is a leg of $M\backslash y_k$, then $L_{k-1}\cup \{x_k, z\}$ is a $4$-circuit of $M$. But $L_{k-1}\cup \{x_k, y_k\}$ is also a $4$-circuit of $M$ and so, as $M$ has no triangles, $L_{k-1}\cup \{x_k, y_k, z\}$ is a $U_{3, 5}$-restriction of $M$, contradicting (\ref{lem:no_u35_spikes}). Thus, $\{x_k, z\}$ is not a leg of $M\backslash y_k$. In turn, this implies that $k\neq r$; otherwise, $M\backslash y_k$ is not a spike minor. Therefore $k=r-1$. Let $x_r$ be the unique element of a $E(M) - L_1 \cup L_2 \cup \cdots L_k \cup \{z\}$. Now, $M \backslash y_k$ has a full leg containing two of the elements $x_k$, $x_r$, and $z$. This leg is not $\{x_k, z\}$, and, since $M \backslash z$ does not have a $4$-element circuit containing $\{x_k, x_r\}$, it is not $\{x_k,x_r\}$ either. Thus, $\{x_r, z\}$ is a leg of $M \backslash y_k$. 

The set $\{x_k,y_k,x_r\}$ is a triad of $M \backslash z$, so $\{x_k,y_k,x_r,z\}$ is a cocircuit of $M$. Suppose $r(\{x_k, y_k, x_r, z\})=4$. Then, as $r(L_i\cup \{x_r, z\})=3$ and $r(L_i\cup L_k)=3$ for all $i\in \{1, 2, \ldots, k-1\}$, it follows by submodularity that
\begin{align*}
r(L_i\cup L_k\cup \{x_r, z\}) & \le r(L_i\cup \{x_r, z\}) + r(L_i\cup L_k) - r(L_i) \\
& = 4.
\end{align*}
But then $L_i\subseteq \cl(L_{r-1}\cup \{x_r, z\})$ for all $i$, and so $r(M)=4$, a contradiction. Therefore $\{x_{r-1}, y_{r-1}, x_r, z\}$ is a circuit of $M$. Hence, by Lemma~\ref{pairs}, $M$ is a rank-$r$ tipless spike. This last contradiction proves (\ref{rank bound}).

\begin{sublemma}
\label{r4minors} 
$r(M)=r^*(M)=4$. Furthermore, for all $e\in E(M)$, we have $M/e$ is a tipped $3$-spike and, dually, $M\backslash e$ is a single-element deletion of a rank-$4$ tipless spike.
\end{sublemma}

By duality and (\ref{rank bound}), we may assume that $r(M)=4$. As $M$ has no triangles, it follows by Lemma~\ref{lemma: spike degeneracy} that $M/e$ is a simple restriction of a rank-$3$ tipped spike. This implies that $r^*(M) \leq 4$, so $r^*(M)=4$. This proves (\ref{r4minors}).

\begin{sublemma}
\label{cor:buddy_lemma}
For each $x\in E(M)$, there exists a unique element $y\in E(M)-\{x\}$ with the property that $\{x, y\}$ is a subset of three distinct $4$-circuits of $M$.
\end{sublemma}

The existence of $y$ follows from the fact that $M/x$ is a tipped rank-$3$ spike. To establish uniqueness, let $E(M)=\{1, 2, 3, 4, 5, 6, 7, 8\}$, and suppose to the contrary that $\{1, 2\}$ and $\{1, 3\}$ are each contained in three distinct $4$-circuits of $M$. Then, by considering $M/1$, it follows, up to labelling, that $M$ has the $4$-circuits
$$\{1, 2, 3, 4\},\, \{1, 2, 5, 6\},\, \{1, 2, 7, 8\},\, \{1, 3, 5, 7\}, \{1, 3, 6, 8\}.$$
We aim to construct enough $4$-circuits of $M$ to show there exists a partition of $E(M)$ into $2$-element subsets such that the union of any two of these subsets is a circuit of $M$. Since $M$ has no triangles or $U_{3,5}$-restrictions, this suffices, by Lemma~\ref{pairs}, to show that $M$ is a rank-$4$ spike, thereby obtaining a contradiction.

By (\ref{r4minors}), $M\backslash 1$ is a single-element deletion of a rank-$4$ tipless spike. We break the rest of the proof of (\ref{cor:buddy_lemma}) into three cases: (i) either $\{2\}$ or $\{3\}$ is a half leg of $M\backslash 1$; (ii) $\{2, 3\}$ is a full leg of $M\backslash 1$; and (iii) the elements $2$ and $3$ are on distinct full legs of $M\backslash 1$. For convenience, we denote, for all $i\in\{2, 3, \ldots, 8\}$, the leg of $M\backslash 1$ containing $i$ by $L_i$.

Consider (i). Without loss of generality, we may assume that $\{2\}$ is a half leg of $M\backslash 1$. If the elements $3$ and $4$ are on distinct legs of $M\backslash 1$, then there is a $4$-circuit of $M\backslash 1$ containing either $\{3, 5, 7\}$ or $\{3, 6, 8\}$. But $\{1, 3, 5, 7\}$ and $\{1, 3, 6, 8\}$ are $4$-circuits of $M$, and so $M$ has a $U_{3, 5}$-restriction, contradicting~(\ref{lem:no_u35_spikes}). Hence $\{3, 4\}$ is a leg of $M\backslash 1$.

Since $\{1, 3, 5, 7\}$ is a circuit of $M$ and $M$ has no $U_{3, 5}$-restriction, $\{5,7\}$ is not a leg of $M\backslash 1$. Hence the set of full legs of $M\backslash 1$ is either $\{\{3, 4\}, \{5, 6\}, \{7, 8\}\}$ or $\{\{3, 4\}, \{5, 8\}, \{6, 7\}\}$. If the first possibility holds, then
$$\{\{1, 2\}, \{3, 4\}, \{5, 6\}, \{7, 8\}\}$$
is a partition of $E(M)$ whose pairwise unions are all $4$-circuits of $M$, and so $M$ is a rank-$4$ spike, a contradiction. Thus the second possibility holds and the set of $4$-circuits of $M$ include
\begin{align*}
\{1, 2, 3, 4\},\, \{1, 2, 5, 6\},\, \{1, 2, 7, 8\},\, \{1, 3, 5, 7\}, \\
\{1, 3, 6, 8\},\, \{3, 4, 5, 8\},\, \{3, 4, 6, 7\},\, \{5, 6, 7, 8\}.
\end{align*}

Now consider $M\backslash 5$. First, observe that a union of two legs of $M \backslash 5$ intersects every other $4$-circuit of $M \backslash 5$ in exactly two elements. In particular, this implies that $\{1,2,7,8\}$ and $\{3,4,6,7\}$ are traversals. The element in the half leg of $M \backslash 5$ appears in every traversal, so this element is $7$. Since $\{1,2,3,4\}$ and $\{1,3,6,8\}$ do not contain $7$, these circuits are unions of two legs. Hence, $M \backslash 5$ has full legs $\{1,3\}$, $\{2,4\}$, and $\{6,8\}$. Thus, $\{2,4,6,8\}$ is also a $4$-circuit of $M$.

Similarly, $\{1,2,7,8\}$ and $\{3,4,5,8\}$ are traversals of $M \backslash 6$, and $\{8\}$ is the half leg of $M \backslash 6$. The circuits $\{1,2,3,4\}$ and $\{1,3,5,7\}$ are unions of two legs, so $M \backslash 6$ has full legs $\{1,3\}$, $\{2,4\}$, and $\{5,7\}$. Therefore, $\{2,4,5,7\}$ is a circuit of $M$. But this implies that $M$ is a rank-$4$ spike whose set of legs is $\{\{1,3\},\{2,4\},\{5,7\},\{6,8\}\}$, a contradiction. Thus, neither $\{2\}$ nor $\{3\}$ is a half leg of $M \backslash 1$.

Consider (ii), and suppose that $\{2, 3\}$ is a full leg of $M\backslash 1$. If $L_4$ is full leg of $M\backslash 1$, then, as $\{1, 2, 3, 4\}$ is a $4$-circuit of $M$, we have that $\{1, 2, 3\}\cup L_4$ induces a $U_{3, 5}$-restriction of $M$, contradicting (\ref{lem:no_u35_spikes}). Therefore $\{4\}$ is the half leg of $M\backslash 1$. If either $\{6, 8\}$ or $\{7, 8\}$ is a full leg of $M\backslash 1$, then, as $\{1, 2, 7, 8\}$ and $\{1, 3, 6, 8\}$ are $4$-circuits of $M$, it follows that either $\{1, 2, 3, 6, 8\}$ or $\{1, 2, 3, 7, 8\}$ induce a $U_{3,5}$-restriction of $M$. This contradiction to (\ref{lem:no_u35_spikes}) implies that $\{5, 8\}$ and $\{6, 7\}$ are legs of $M\backslash 1$. The list of $4$-circuits of $M$ now includes
\begin{align*}
\{1, 2, 3, 4\},\, \{1, 2, 5, 6\},\, \{1, 2, 7, 8\},\, \{1,3,5,7\},\\
\{1, 3, 6, 8\},\, \{2, 3, 5, 8\},\, \{2, 3, 6, 7\},\, \{5, 6, 7, 8\}.
\end{align*}

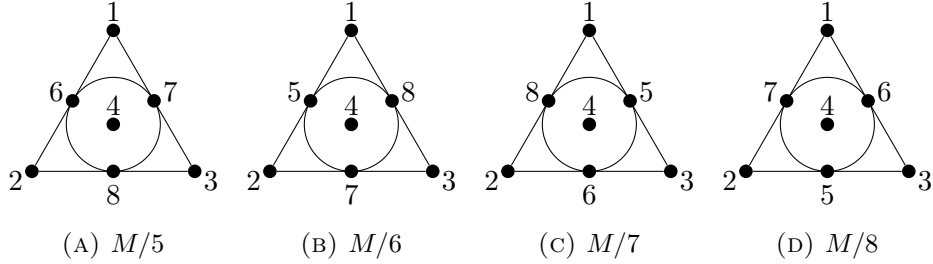
\begin{figure}
	\centering
	\begin{subfigure}{0.24\textwidth}
	\centering
	\begin{tikzpicture}[scale=1.25]
	\begin{scope}[every node/.style=element]
	\node (x1) at (-30:1) {};
	\node (x2) at (90:1) {};
	\node (x3) at (-150:1) {};
	\coordinate (c) at (0,0) {};	
	\node (y1) at ($(x1)!0.5!(x2)$) {};
	\node (y2) at ($(x2)!0.5!(x3)$) {};
	\node (y3) at ($(x3)!0.5!(x1)$) {};
	\draw (x1) to (x2) to (x3) to (x1);
	\node (c) at (0,0) {};
	\draw circle (0.5);
	\end{scope}
	\begin{scope}
	\node (x1) at (-30:1.2) {$3$};
	\node (x2) at (90:1.2) {$1$};
	\node (x3) at (-150:1.2) {$2$};
	\coordinate (y1) at ($(x1)!0.5!(x2)$) {};
	\node at ($(y1) + (30:0.1)$) {$7$};
	\coordinate (y2) at ($(x2)!0.5!(x3)$) {};
	\node at ($(y2) + (150:0.1)$) {$6$};
	\coordinate (y3) at ($(x3)!0.5!(x1)$) {};
	\node at ($(y3) + (-90:0.15)$) {$8$};
	\node (c) at (0,0.2) {$4$};
	\end{scope}
	\end{tikzpicture}
	\subcaption{$M / 5$}
	\end{subfigure}
	\begin{subfigure}{0.24\textwidth}
	\centering
	\begin{tikzpicture}[scale=1.25]
	\begin{scope}[every node/.style=element, shift={(2.5,0)}]
	\node (x1) at (-30:1) {};
	\node (x2) at (90:1) {};
	\node (x3) at (-150:1) {};
	\coordinate (c) at (0,0) {};	
	\node (y1) at ($(x1)!0.5!(x2)$) {};
	\node (y2) at ($(x2)!0.5!(x3)$) {};
	\node (y3) at ($(x3)!0.5!(x1)$) {};
	\draw (x1) to (x2) to (x3) to (x1);
	\node (c) at (0,0) {};
	\draw circle (0.5);
	\end{scope}
	\begin{scope}[shift={(2.5,0)}]
	\node (x1) at (-30:1.2) {$3$};
	\node (x2) at (90:1.2) {$1$};
	\node (x3) at (-150:1.2) {$2$};
	\coordinate (y1) at ($(x1)!0.5!(x2)$) {};
	\node at ($(y1) + (30:0.1)$) {$8$};
	\coordinate (y2) at ($(x2)!0.5!(x3)$) {};
	\node at ($(y2) + (150:0.1)$) {$5$};
	\coordinate (y3) at ($(x3)!0.5!(x1)$) {};
	\node at ($(y3) + (-90:0.15)$) {$7$};
	\node (c) at (0,0.2) {$4$};
	\end{scope}
	\end{tikzpicture}
	\subcaption{$M / 6$}
	\end{subfigure}
	\begin{subfigure} {0.24\textwidth}
	\centering
	\begin{tikzpicture}[scale=1.25]
	\begin{scope}[every node/.style=element, shift={(5,0)}]
	\node (x1) at (-30:1) {};
	\node (x2) at (90:1) {};
	\node (x3) at (-150:1) {};
	\coordinate (c) at (0,0) {};	
	\node (y1) at ($(x1)!0.5!(x2)$) {};
	\node (y2) at ($(x2)!0.5!(x3)$) {};
	\node (y3) at ($(x3)!0.5!(x1)$) {};
	\draw (x1) to (x2) to (x3) to (x1);
	\node (c) at (0,0) {};
	\draw circle (0.5);
	\end{scope}
	\begin{scope}[shift={(5,0)}]
	\node (x1) at (-30:1.2) {$3$};
	\node (x2) at (90:1.2) {$1$};
	\node (x3) at (-150:1.2) {$2$};
	\coordinate (y1) at ($(x1)!0.5!(x2)$) {};
	\node at ($(y1) + (30:0.1)$) {$5$};
	\coordinate (y2) at ($(x2)!0.5!(x3)$) {};
	\node at ($(y2) + (150:0.1)$) {$8$};
	\coordinate (y3) at ($(x3)!0.5!(x1)$) {};
	\node at ($(y3) + (-90:0.15)$) {$6$};
	\node (c) at (0,0.2) {$4$};
	\end{scope}
	\end{tikzpicture}
	\subcaption{$M / 7$}
	\end{subfigure}
	\begin{subfigure}{0.24\textwidth}
	\centering
	\begin{tikzpicture}[scale=1.25]
	\begin{scope}[every node/.style=element, shift={(7.5,0)}]
	\node (x1) at (-30:1) {};
	\node (x2) at (90:1) {};
	\node (x3) at (-150:1) {};
	\coordinate (c) at (0,0) {};	
	\node (y1) at ($(x1)!0.5!(x2)$) {};
	\node (y2) at ($(x2)!0.5!(x3)$) {};
	\node (y3) at ($(x3)!0.5!(x1)$) {};
	\draw (x1) to (x2) to (x3) to (x1);
	\node (c) at (0,0) {};
	\draw circle (0.5);
	\end{scope}
	\begin{scope}[shift={(7.5,0)}]
	\node (x1) at (-30:1.2) {$3$};
	\node (x2) at (90:1.2) {$1$};
	\node (x3) at (-150:1.2) {$2$};
	\coordinate (y1) at ($(x1)!0.5!(x2)$) {};
	\node at ($(y1) + (30:0.1)$) {$6$};
	\coordinate (y2) at ($(x2)!0.5!(x3)$) {};
	\node at ($(y2) + (150:0.1)$) {$7$};
	\coordinate (y3) at ($(x3)!0.5!(x1)$) {};
	\node at ($(y3) + (-90:0.15)$) {$5$};
	\node (c) at (0,0.2) {$4$};
	\end{scope}
	\end{tikzpicture}
	\subcaption{$M / 8$}
	\end{subfigure}
	\caption{Known circuits of $M/5$, $M/6$, $M/7$, and $M/8$ in the proof of (\ref{cor:buddy_lemma}).}
	\label{fig: Em1}
\end{figure}

Consider each of the contractions $M/5$, $M/6$, $M/7$, and $M/8$. Figure~\ref{fig: Em1} shows geometric representations of these matroids, including only the $3$-element circuits implied by the above list. By (\ref{r4minors}), each matroid is a rank-$3$ tipped spike, and so there is at least one additional $3$-element circuit in each containing $4$. Thus, the rank-$3$ tipped spike $M/5$ implies that at least one of $\{1, 4, 5, 8\}$, $\{2, 4, 5, 7\}$, and $\{3, 4, 5, 6\}$ is a $4$-circuit of $M$. Similarly, $M/6$ implies that at least one of $\{1, 4, 6, 7\}$, $\{2, 4, 6, 8\}$, and $\{3, 4, 5, 6\}$ is a $4$-circuit of $M$, the matroid $M/7$ implies that at least one of $\{1, 4, 6, 7\}$, $\{2, 4, 5, 7\}$,  and $\{3, 4, 7, 8\}$ is a $4$-circuit of $M$, and $M/8$ implies that at least one of $\{1, 4, 5, 8\}$, $\{2, 4, 6, 8\}$, and $\{3, 4, 7, 8\}$ is a $4$-circuit of $M$.

If $\{1, 4, 5, 8\}$ and $\{1, 4, 6, 7\}$ are $4$-circuits of $M$, then $M$ is a rank-$4$ spike whose set of legs is $\{\{1, 4\}, \{2, 3\}, \{5, 8\}, \{6, 7\}\}$, a contradiction.  Similarly, if $\{2, 4, 6, 8\}$ and $\{2, 4, 5, 7\}$ are $4$-circuits of $M$, then $M$ is a rank-$4$ spike whose set of legs is $\{\{1, 3\}, \{2, 4\}, \{5, 7\}, \{6, 8\}\}$ and, if $\{3, 4, 5, 6\}$ and $\{3, 4, 7, 8\}$ are $4$-circuits of $M$, then $M$ is a rank-$4$ spike whose set of legs is $\{\{1, 2\}, \{3, 4\}, \{5, 6\}, \{7, 8\}\}$. We next check  possible combinations of $4$-circuits that avoid these three pairs of $4$-element sets.

If $\{1, 4, 5, 8\}$ is a $4$-circuit of $M$, then either $\{2, 4, 6, 8\}$ and $\{3, 4, 7, 8\}$ are $4$-circuits of $M$, or $\{3, 4, 5, 6\}$ and $\{2, 4, 5, 7\}$ are $4$-circuits of $M$ to ensure that both $M/6$ and $M/7$ are rank-$3$ tipped spikes. If the first instance holds, then $M$ is a rank-$4$ spike whose set of legs is $\{\{1, 5\}, \{2, 6\}, \{3, 7\}, \{4, 8\}\}$ and, if the second instance holds, then $M$ is a rank-$4$ spike whose set of legs is $\{\{1, 8\}, \{2, 7\}, \{3, 6\}, \{4, 5\}\}$. Thus, $\{1, 4, 5, 8\}$ is not a $4$-circuit of $M$.

If $\{3, 4, 7, 8\}$ is a $4$-circuit of $M$, then $\{2, 4, 5, 7\}$ and $\{1, 4, 6, 7\}$ are also $4$-circuits of $M$ to ensure that both $M/5$ and $M/6$ are rank-$3$ tipped spikes. But then $M$ is a rank-$4$ spike whose set of legs is $\{\{1, 6\}, \{2, 5\}, \{3, 8\}, \{4, 7\}\}$, a contradiction. Therefore, $\{3, 4, 7, 8\}$ is not a $4$-circuit of $M$. Hence, to ensure $M/8$ is a rank-$3$ tipped spike, $\{2, 4, 6, 8\}$ is a $4$-circuit of $M$, in which case $\{3, 4, 5, 6\}$ and $\{1, 4, 6, 7\}$ are also $4$-circuits of $M$ so that both $M/5$ and $M/7$ are rank-$3$ tipped spikes. In this last instance, $M$ is a rank-$4$ spike whose set of legs is $\{\{1, 7\}, \{2, 8\}, \{3, 5\}, \{4, 6\}\}$. This last contradiction implies that $\{2, 3\}$ is not a leg of $M\backslash 1$.

Lastly, consider (iii), and suppose that elements $2$ and $3$ are on distinct full legs of $M\backslash 1$. First, assume that $\{4\}$ is the half leg of $M\backslash 1$. We consider the possibilities for the element on the same leg as $2$. This element is not $3$ or $4$. If $\{2,5\}$ or $\{2,6\}$ is a leg of $M \backslash 1$ then the circuit $\{1,2,5,6\}$ implies that either $L_2 \cup L_5 \cup \{1\}$ or $L_2 \cup L_6 \cup \{1\}$ induces a $U_{3,5}$-restriction of $M$, a contradiction. In the same way, the circuit $\{1,2,7,8\}$ implies that neither $\{2,7\}$ nor $\{2,8\}$ is a leg of $M \backslash 1$. Thus, there are no possible element which can be on the same leg as $2$, and so $\{4\}$ is not the half leg of $M \backslash 1$.

If either $\{2, 4\}$ or $\{3, 4\}$ is a leg of $M\backslash 1$, then $L_2\cup L_3\cup \{1\}$ induces a $U_{3, 5}$-restriction of $M$ with the circuit $\{1,2,3,4\}$. Therefore $2$, $3$, and $4$ are on distinct full legs of $M\backslash x$. Observing that the elements $5$, $6$, $7$, and $8$ are not distinguishable in our initial assignment of circuits, we may assume without loss of generality that $\{6\}$ is the half leg of $M\backslash 1$. Moreover, if $\{2, 3, 7, 8\}$, $\{2, 3, 5, 7\}$, or $\{2, 4, 7, 8\}$ is a $4$-circuit of $M$, then $r(\{1, 2, 3, 7, 8\})=3$, $r(\{1, 2, 3, 5, 7\})=3$, or $r(\{1, 2, 4, 7, 8\})=3$, a contradiction. Hence $M\backslash 1$ has legs $\{2, 5\}$, $\{3, 8\}$, and $\{4, 7\}$, and the list of $4$-circuits of $M$ now include
\begin{align*}
\{1, 2, 3, 4\},\, \{1, 2, 5, 6\},\, \{1, 2, 7, 8\},\, \{1, 3, 5, 7\}, \\
\{1, 3, 6, 8\},\, \{2, 3, 5, 8\},\, \{2, 4, 5, 7\},\, \{3, 4, 7, 8\}.
\end{align*}

Consider $M/4$. The geometric representation of $M/4$ is shown in Figure~\ref{subfig: m/4}, including only the $3$-element circuits implied by the above list. Since $M/4$ is a rank-$3$ tipped spike, it has at least one of the circuits $\{1,6,7\}$, $\{2,6,8\}$, and $\{3,5,6\}$. Hence, at least one of $\{1, 4, 6, 7\}$, $\{2, 4, 6, 8\}$, and $\{3, 4, 5, 6\}$ is a $4$-circuit of $M$. If $\{1, 4, 6, 7\}$ is a $4$-circuit of $M$, then $M$ is a rank-$4$ spike whose set of legs is $\{\{1, 6\}, \{2, 5\}, \{3, 8\}, \{4, 7\}\}$. Hence, either $\{2, 4, 6, 8\}$ or $\{3, 4, 5, 6\}$ is a $4$-circuit of $M$. Consider $M/6$. Depending on whether $\{2, 4, 6, 8\}$ or $\{3, 4, 5, 6\}$ is a $4$-circuit of $M$, the possible geometric representations of $M/6$ are shown in Figures~\ref{subfig: m/6 2468} and \ref{subfig: m/6 3456}, respectively (including only the known $3$-element circuits). Regardless of whether $\{2, 4, 6, 8\}$ or $\{3, 4, 5, 6\}$ is a $4$-circuit of $M$, either $\{2,3,7\}$ or $\{5,7,8\}$ is a circuit of $M / 6$, so either $\{2, 3, 6, 7\}$ or $\{5, 6, 7, 8\}$ is a $4$-circuit of $M$. If both $\{2, 4, 6, 8\}$ and $\{5, 6, 7, 8\}$ are $4$-circuits of $M$, then $M$ is a rank-$4$ spike whose set of legs is $\{\{1, 3\}, \{2, 4\}, \{5, 7\}, \{6, 8\}\}$. Furthermore, if $\{3, 4, 5, 6\}$ and $\{5, 6, 7, 8\}$ are $4$-circuits of $M$, then $M$ is a rank-$4$ spike whose set of legs is $\{\{1, 2\}, \{3, 4\}, \{5, 6\}, \{7, 8\}\}$. Therefore, $\{5, 6, 7, 8\}$ is not a $4$-circuit of $M$, and so $\{2, 3, 6, 7\}$ is a $4$-circuit of $M$. The list of $4$-circuits of $M$ now include
\begin{align*}
& \{1, 2, 3, 4\},\, \{1, 2, 5, 6\},\, \{1, 2, 7, 8\},\, \{1, 3, 5, 7\},\, \{1, 3, 6, 8\}, \\
& \{2, 3, 5, 8\},\, \{2, 4, 5, 7\},\, \{3, 4, 7, 8\},\, \{2,3,6,7\},
\end{align*}
together with at least one of $\{2, 4, 6, 8\}$ and $\{3, 4, 5, 6\}$.

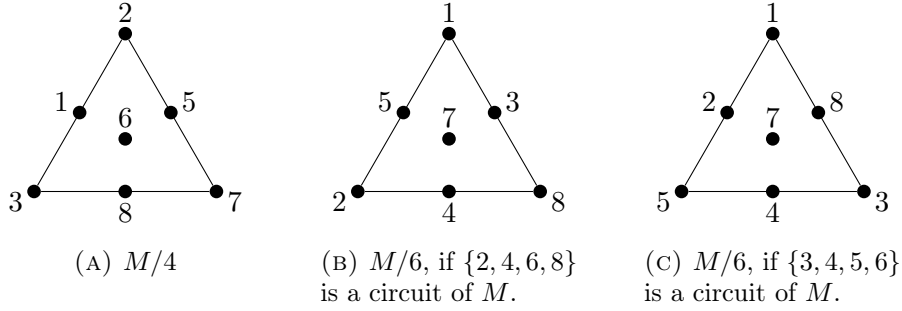
\begin{figure}
	\centering
	\begin{subfigure}[t]{0.27\textwidth}
	\centering
	\begin{tikzpicture}[scale=1.4]
	\centering
	\begin{scope}[every node/.style=element]
	\node (x1) at (-30:1) {};
	\node (x2) at (90:1) {};
	\node (x3) at (-150:1) {};
	\coordinate (c) at (0,0) {};	
	\node (y1) at ($(x1)!0.5!(x2)$) {};
	\node (y2) at ($(x2)!0.5!(x3)$) {};
	\node (y3) at ($(x3)!0.5!(x1)$) {};
	\draw (x1) to (x2) to (x3) to (x1);
	\node (c) at (0,0) {};
	\end{scope}
	\begin{scope}
	\node (x1) at (-30:1.2) {$7$};
	\node (x2) at (90:1.2) {$2$};
	\node (x3) at (-150:1.2) {$3$};
	\coordinate (y1) at ($(x1)!0.5!(x2)$) {};
	\node at ($(y1) + (30:0.1)$) {$5$};
	\coordinate (y2) at ($(x2)!0.5!(x3)$) {};
	\node at ($(y2) + (150:0.1)$) {$1$};
	\coordinate (y3) at ($(x3)!0.5!(x1)$) {};
	\node at ($(y3) + (-90:0.1)$) {$8$};
	\node (c) at (0,0.2) {$6$};
	\end{scope}
	\end{tikzpicture}
	\subcaption{$M / 4$} \label{subfig: m/4}
	\end{subfigure}
	\hspace{0.05\textwidth}
	\begin{subfigure}[t]{0.27\textwidth}
	\centering
	\begin{tikzpicture} [scale = 1.4]
	\begin{scope}[every node/.style=element, shift={(6,0)}]
	\node (x1) at (-30:1) {};
	\node (x2) at (90:1) {};
	\node (x3) at (-150:1) {};
	\coordinate (c) at (0,0) {};	
	\node (y1) at ($(x1)!0.5!(x2)$) {};
	\node (y2) at ($(x2)!0.5!(x3)$) {};
	\node (y3) at ($(x3)!0.5!(x1)$) {};
	\draw (x1) to (x2) to (x3) to (x1);
	\node (c) at (0,0) {};
	\end{scope}
	\begin{scope}[shift={(6,0)}]
	\node (x1) at (-30:1.2) {$8$};
	\node (x2) at (90:1.2) {$1$};
	\node (x3) at (-150:1.2) {$2$};
	\coordinate (y1) at ($(x1)!0.5!(x2)$) {};
	\node at ($(y1) + (30:0.1)$) {$3$};
	\coordinate (y2) at ($(x2)!0.5!(x3)$) {};
	\node at ($(y2) + (150:0.1)$) {$5$};
	\coordinate (y3) at ($(x3)!0.5!(x1)$) {};
	\node at ($(y3) + (-90:0.1)$) {$4$};
	\node (c) at (0,0.2) {$7$};
	\end{scope}
	\end{tikzpicture}
	\subcaption{$M / 6$, if $\{2,4,6,8\}$ is a circuit of $M$.} \label{subfig: m/6 2468}
	\end{subfigure}
	\hspace{0.05\textwidth}
	\begin{subfigure}[t]{0.27\textwidth}
	\centering
	\begin{tikzpicture} [scale = 1.4]
	\begin{scope}[every node/.style=element, shift={(3,0)}]
	\node (x1) at (-30:1) {};
	\node (x2) at (90:1) {};
	\node (x3) at (-150:1) {};
	\coordinate (c) at (0,0) {};	
	\node (y1) at ($(x1)!0.5!(x2)$) {};
	\node (y2) at ($(x2)!0.5!(x3)$) {};
	\node (y3) at ($(x3)!0.5!(x1)$) {};
	\draw (x1) to (x2) to (x3) to (x1);
	\node (c) at (0,0) {};
	\end{scope}
	\begin{scope}[shift={(3,0)}]
	\node (x1) at (-30:1.2) {$3$};
	\node (x2) at (90:1.2) {$1$};
	\node (x3) at (-150:1.2) {$5$};
	\coordinate (y1) at ($(x1)!0.5!(x2)$) {};
	\node at ($(y1) + (30:0.1)$) {$8$};
	\coordinate (y2) at ($(x2)!0.5!(x3)$) {};
	\node at ($(y2) + (150:0.1)$) {$2$};
	\coordinate (y3) at ($(x3)!0.5!(x1)$) {};
	\node at ($(y3) + (-90:0.1)$) {$4$};
	\node (c) at (0,0.2) {$7$};
	\end{scope}
	\end{tikzpicture}
	\subcaption{$M / 6$, if $\{3,4,5,6\}$ is a circuit of $M$.} \label{subfig: m/6 3456}
	\end{subfigure}
	\caption{Known circuits of $M / 4$ and $M / 6$ in the proof of~(\ref{cor:buddy_lemma}).}
	\label{fig:case3a}
\end{figure}

Now consider $M\backslash 2$. For each choice of $L_1$, there is exactly one possibility for $M\backslash 2$. In particular, if $L_1$ is the half leg, then the $4$-circuit $\{1, 3, 5, 7\}$ is a traversal of $M\backslash 2$ and so, as $\{1, 3, 6, 8\}$ and $\{3, 4, 7, 8\}$ are $4$-circuits of $M$, it is easily checked that $M\backslash 2$ has legs $\{1\}$, $\{3, 4\}$, $\{5, 6\}$, $\{7, 8\}$. A similar analysis shows that if $\{1, 3\}$, $\{1, 4\}$, $\{1, 5\}$, $\{1, 6\}$, $\{1, 7\}$ or $\{1, 8\}$ is a leg of $M\backslash 2$, then the set of legs of $M\backslash 2$ is
\begin{align*}
\{\{1, 3\}, \{5, 7\}, \{6, 8\}, \{4\}\},\, \{\{1, 4\}, \{3\}, \{5, 8\}, \{6, 7\}\}, \\
\{\{1, 5\}, \{3, 7\}, \{6\}, \{4, 8\}\},\, \{\{1, 6\}, \{3, 8\}, \{5\}, \{4, 7\}\}, \\
\{\{1, 7\}, \{3, 5\}, \{4, 6\}, \{8\}\},\, \{\{1, 8\}, \{3, 6\}, \{4, 5\}, \{7\}\},
\end{align*}
respectively. If $L_1\in \{\{1\}, \{1, 3\}, \{1, 4\}, \{1, 6\}, \{1, 7\}\}$, then either $\{1, 4, 6, 7\}$ or $\{5, 6, 7, 8\}$ is a $4$-circuit of $M$, and we have previously seen that $M$ is a rank-$4$ spike. Otherwise, $L_1\in \{\{1, 5\}, \{1, 8\}\}$, in which case $\{1, 4, 5, 8\}$ is a $4$-circuit of $M$. If $\{2, 4, 6, 8\}$ is a $4$-circuit of $M$, then $M$ is a rank-$4$ spike whose set of legs is $\{\{1, 5\}, \{2, 6\}, \{3, 7\}, \{4, 8\}\}$. On the other hand, if $\{3, 4, 5, 6\}$ is a $4$-circuit of $M$, then $M$ is a rank-$4$ spike whose set of legs is $\{\{1, 8\}, \{2, 7\}, \{3, 6\}, \{4, 5\}\}$. This last contradiction proves (\ref{cor:buddy_lemma}).

\begin{sublemma}
\label{lem:unique_traversal_intersection}
For each $x \in E(M)$, any two distinct traversals of $M\backslash x$ meet in exactly one element.
\end{sublemma}

Let $x\in E(M)$, and consider $M\backslash x$. Since $M\backslash x$ is a restriction of a rank-$4$ (tipless) spike, $M\backslash x$ has exactly three full legs. Denote the legs of $M\backslash x$ by $\{a, b\}$, $\{c, d\}$, $\{e, f\}$, and $\{y\}$. Each traversal of $M \backslash x$ contains $y$, so any two distinct traversals meet in at least one element. Suppose $M \backslash x$ has two distinct traversals that meet in two elements. We may assume these traversals are $\{a,c,e,y\}$ and $\{b,d,e,y\}$. Since $\{a, b, c, d\}$, $\{a, b, e, f\}$, and $\{c, d, e, f\}$ are $4$-circuits of $M$, the sets $\{x, y, e, f\}$, $\{x, y, c, d\}$, and $\{x, y, a, b\}$ are hyperplanes of $M^*$. In turn, by (\ref{lem:no_triangles_or_triads_spikes}), this implies that these three hyperplanes are all $4$-circuits of $M^*$. Additionally, the traversals $\{a,c,e,y\}$ and $\{b,d,e,y\}$ imply that $\{x, f, b, d\}$ and $\{x, f, a, c\}$ are hyperplanes of $M^*$, and thus $4$-circuits of $M^*$. But then $x$ is in three distinct $4$-circuits of $M^*$ containing $y$ as well as three distinct $4$-circuits of $M^*$ containing $f$, contradicting the dual of (\ref{cor:buddy_lemma}). This proves (\ref{lem:unique_traversal_intersection}).

\begin{sublemma}
\label{lem:spike_representations}
If $x\in E(M)$, then there is a unique choice for the set of legs of $M\backslash x$. Furthermore, each of the following properties hold:
\begin{enumerate}[{\rm (i)}]
\item If $x\in E(M)$, then there is a unique element $y\in E(M)$ such that $\{x\}$ is the half leg of $M\backslash y$ and $\{y\}$ is the half leg $M\backslash x$.

\item If $\{x', y'\}\subset E(M)$, then there are at most two elements $z_1, z_2\in E(M)$ such that $\{x', y'\}$ is a full leg of $M\backslash z_1$ and $M\backslash z_2$.
\end{enumerate}
\end{sublemma}

Let $x\in E(M)$ and suppose that a set of legs for $M\backslash x$ is $\{\{a, b\}, \{c, d\}, \{e, f\}, \{y\}\}$. If there is another choice for a set of legs of $M\backslash x$, then exactly two of the $4$-circuits $\{a, b, c, d\}$, $\{c, d, e, f\}$, and $\{a, b, e, f\}$ are traversals of $M\backslash x$ in this other choice, contradicting (\ref{lem:unique_traversal_intersection}). Thus the set of legs of $M\backslash x$ is unique.

To prove (i), observe that $\{x, y, e, f\}$, $\{x, y, c, d\}$, and $\{x, y, a, b\}$ are hyperplanes of $M^*$ and so, by (\ref{lem:no_triangles_or_triads_spikes}), $4$-circuits of $M^*$. Thus $y$ is in three distinct $4$-circuits of $M^*$ containing $x$. Consider $M\backslash y$. If $\{z\}$ is the half leg of $M\backslash y$, where $z\in E(M)-\{y\}$, then, by applying the same argument, $z$ is in three distinct $4$-circuits of $M^*$ containing $y$. By (\ref{cor:buddy_lemma}), we deduce that $z=x$, thereby proving (i). 

Lastly, we prove (ii). Take a pair which is a full leg of some single-element deletion of $M$. We may assume this pair is $\{a,b\}$. Suppose there is a single-element deletion of $M$ other than $M \backslash x$ and $M \backslash y$ for which $\{a,b\}$ is a full leg. Without loss of generality, let this single-element deletion be $M \backslash c$. Since $\{a, b, e, f\}$ is a $4$-circuit of $M\backslash c$, the set $\{e, f\}$ is a leg of $M\backslash c$. Furthermore, by~(i), neither $\{x\}$ nor $\{y\}$ is the half leg of $M\backslash c$, and so $\{x, y\}$ is the final full leg of $M\backslash c$, and $\{d\}$ is the half leg of $M\backslash c$. Now consider $M\backslash e$. Since $\{a, b, c, d\}$ and $\{a, b, x, y\}$ are $4$-circuits of $M\backslash e$, it follows by (\ref{lem:unique_traversal_intersection}) that neither circuit is a traversal of $M\backslash e$. So $M\backslash e$ has legs $\{a, b\}$, $\{c, d\}$, and $\{x, y\}$. But then $M$ is a rank-$4$ spike whose set of legs is $\{\{a, b\}, \{c, d\}, \{e, f\}, \{x, y\}\}$, a contradiction. This completes the proof of (ii) and, therefore, (\ref{lem:spike_representations}).

\begin{sublemma}
\label{lem:two_traversals}
For each $x\in E(M)$, the matroid $M\backslash x$ has at least two traversals.
\end{sublemma}

Let $x\in E(M)$, and suppose that $M\backslash x$ has legs $\{a, b\}$, $\{c, d\}$, $\{e, f\}$, and $\{y\}$. Consider each of the minors $M\backslash a$, $M\backslash b$, $M\backslash c$, $M\backslash d$, $M\backslash e$, and $M\backslash f$ as a restriction of a rank-$4$ spike. By~(\ref{lem:spike_representations}), $y$ is in a full leg of each of these matroids and $\{x, y\}$ is a leg of at most two of these matroids. Hence, in at least four of $M\backslash a$, $M\backslash b$, $M\backslash c$, $M\backslash d$, $M\backslash e$, and $M\backslash f$, we have that $y$ is in a $4$-circuit not containing $x$. Since such a $4$-circuit contains three elements of $\{a, b, c, d, e, f\}$, it follows that $M$ has two $4$-circuits containing $y$ but not $x$. As $\{y\}$ is the half leg of $M\backslash x$, this implies that $M\backslash x$ has at least two traversals. This proves (\ref{lem:two_traversals}).

We now complete the proof of the lemma. Let $x \in E(M)$ and consider $M\backslash x$. Combining~(\ref{lem:unique_traversal_intersection}) and~(\ref{lem:two_traversals}), $M\backslash x$ has exactly two traversals, and these traversals meet in the element lying on the half leg. Thus, for all $x\in E(M)$, we have $M\backslash x\cong P_7^*$ and, by duality, $M/x\cong P_7$. By Lemma~\ref{lem: p8}, $M$ is isomorphic to $P_8$.
\end{proof}

\begin{proof}[Proof of Theorem~\ref{main1}]
Since $\cS$ is closed under duality, Lemma~\ref{lemma: small spike-minors} establishes the excluded minors of $\cS$ of rank or corank at most three. Theorem~\ref{main1} now follows by Lemma~\ref{rank4}.
\end{proof}

\section{Proof of Theorem~\ref{them: 3con}}
\label{proof2}

In this section we prove Theorem~\ref{them: 3con}. Recall that $\cS_3$ denotes the subset of~$\cS$ consisting of the $3$-connected members of $\cS$. A $3$-connected matroid $M$ is {\em minimally not in $\cS_3$} if $M$ is not a spike minor but every $3$-connected minor of $M$ is a spike minor. As in the proof of Theorem~\ref{main1}, we partition the proof of Theorem~\ref{them: 3con} into two parts depending on whether a minimal $3$-connected matroid not in $\cS_3$ has rank at most three or rank and corank at least four. 

\begin{lemma}
Let $M$ be a $3$-connected matroid with rank at most three. Then $M$ is minimally not in $\cS_3$ if and only if $M$ is isomorphic to one of $U_{2, 6}$, $O_7$, $O_7^-$, $P_7^-$, $P_7^=$, and $AG(2, 3)\backslash e$.
\label{most3}
\end{lemma}

\begin{proof}
By considering the excluded minors of $\cS$, it is routine to show that each of the $3$-connected matroids $U_{2, 6}$, $O_7$, $O_7^-$, $P_7^-$, $P_7^=$, and $AG(2, 3)\backslash e$ are minimally not in $\cS_3$. Now suppose that $M$ is a $3$-connected matroid that is minimally not in $\cS_3$ and $r(M)\le 3$. Clearly, $r(M)\ge 2$ and, if $r(M)=2$, then $M$ is isomorphic to $U_{2, 6}$. So assume that $r(M)=3$. Up to duality, we may also assume that $r^*(M) \geq 3$.

\begin{sublemma}
$|E(M)|\geq 7$.
\label{most7}
\end{sublemma}

If $|E(M)|\leq6$, then, by Theorem~\ref{main1}, $M$ has a single-element deletion or contraction isomorphic to one of the matroids $U_{2, 4}\oplus U_{0, 1}$, $U_{1, 3}\oplus U_{1, 1}\oplus U_{0, 1}$, $U_{2, 4}\oplus U_{1, 1}$, and $U_{2, 3}\oplus U_{1, 1}\oplus U_{0, 1}$. But this implies that $M$ has a circuit or cocircuit of size at most two, a contradiction to $3$-connectivity. Thus, (\ref{most7}) holds.

\begin{sublemma}
If $|E(M)|=7$, then $M$ is isomorphic to one of $O_7$, $O_7^{-}$, $P_7^{-}$, and $P_7^{=}$.
\label{seven}
\end{sublemma}

If $M$ is an excluded minor for $\cS$, then, by Theorem~\ref{main1}, $M$ is isomorphic to either $P_7^{-}$ or $P_7^{=}$. On the other hand, if $M$ is not an excluded minor for $\cS$, then there is an element $e\in E(M)$ such that either $M/e$ or $M\backslash e$ is not a spike minor. If $M/e$ is isomorphic to either $U_{1, 3}\oplus U_{1, 3}$ or $\{2\}$-$U_{2, 3}\oplus U_{1, 2}$, then $M$ is $2$-separating, a contradiction. If $M$ has a minor isomorphic to either $U_{1, 3}\oplus U_{1, 1}\oplus U_{0, 1}$ or $U_{2, 3}\oplus U_{1, 1}\oplus U_{0, 1}$, then $M$ has a circuit or cocircuit of size at most two, another contradiction. Therefore, by Theorem~\ref{main1}, $M$ has a minor isomorphic to either $U_{2, 4}\oplus U_{0, 1}$ or $U_{2, 4}\oplus U_{1, 1}$. If $M$ has a minor isomorphic to $U_{2, 4}\oplus U_{0, 1}$, then $M$ has a circuit of size two. This contradiction implies that $M$ consists of a $U_{2, 4}$-restriction with ground set $Y$ and a triad $\{e, f, g\}$. As $M$ has no $U_{2, 6}$-minor, we may assume that $\cl(\{e, g\})\cap Y$ and $\cl(\{f, g\})\cap Y$ are both non-empty. If $\cl(\{e, g\})=\cl(\{f, g\})$, then $M$ is $2$-separating, a contradiction. Thus, $\cl(\{e, g\})$ and $\cl(\{f, g\})$ are distinct. It follows that $M$ is isomorphic to $O_7^-$ if $\cl(\{e,f\}) \cap Y$ is empty, and isomorphic to $O_7$ if $\cl(\{e,f\}) \cap Y$ is non-empty. Thus (\ref{seven}) holds.

We may now assume that $|E(M)|\geq 8$.

\begin{sublemma}
\label{sublem: 3con rk3 a}
$M\backslash e$ is a tipped $3$-spike for all $e\in E(M)$.
\end{sublemma}

If $M$ has a $U_{2,5}$-restriction, then, by Lemma~\ref{segdelete}, deleting an element of this restriction produces a $3$-connected matroid. Furthermore, this matroid has $U_{2, 4}\oplus U_{1, 1}$, an excluded minor of $\cS$, as a minor. Thus, $M$ has no $U_{2,5}$-restriction. Suppose $M\backslash e$ has a $2$-separation $(X, Y)$. Since $M$ has rank~$3$ and is simple, $M|X$ and $M|Y$ are each isomorphic to a rank-$2$ uniform matroid. Furthermore, as $|E(M)|\geq 8$ and $M$ has no $U_{2, 5}$-restriction, either $|X|=4$ or $|Y|=4$. Without loss of generality, say $|X|=4$. Let $y\in Y$. Then, as $M$ is $3$-connected and $|Y\cup \{e\}|\ge 4$, the matroid $M\backslash y$ is $3$-connected. So $M\backslash y$ is a rank-$3$ spike minor. But $(M\backslash y)|X\cong U_{2, 4}$, a contradiction. Hence, $M\backslash e$ is $3$-connected for all $e\in E(M)$, and so $M \backslash e$ is a spike minor. Furthermore, $M \backslash e$ is a tipped spike, since $|E(M)| \geq 8$. This completes the proof of (\ref{sublem: 3con rk3 a}).

As a consequence of (\ref{sublem: 3con rk3 a}), we have that $|E(M)|=8$ and that $M$ has no $U_{2, 4}$-restriction. We next argue that $M$ has none of the matroids $U_{2, 5}$, $U_{3, 5}$, $F_7$, or $F_7^*$ as a minor, thereby showing that $M$ is ternary~\cite{bix79, sey79}.

\begin{sublemma}
\label{sublem: rk3 no U25}
$M$ has no $U_{2, 5}$-minor or $U_{3, 5}$-minor.
\end{sublemma}

By Lemma~\ref{lem: U25 U35}, it suffices to prove that $M$ has no $U_{2,5}$-minor. Suppose to the contrary that $M$ has such a minor. Then there are elements $\{e, f, g\}\in E(M)$ such that $M/e\backslash\{f, g\}\cong U_{2, 5}$. If $M/e$ has a parallel class of size three, then $M$ has a $U_{2, 4}$-restriction, a contradiction. Thus, as $M$ has no $U_{2, 6}$-minor, we deduce that $M/e$ is isomorphic to $\{2, 2\}$-$U_{2, 5}$. This means that $e$ is contained in exactly two triangles of $M$. Now consider $M \backslash e$. This is a tipped $3$-spike, so contains an element $t$ which is contained in three triangles of $M \backslash e$. The elements of $M \backslash e$ can be contained in at most two triangles of $M \backslash e$ which do not contain $t$. However, they cannot then be contained in a triangle containing $e$, as this would imply that $M$ has a $U_{2,4}$-restriction. Therefore, $M\backslash t$ has no element in at least three triangles, contradicting (\ref{sublem: 3con rk3 a}). Thus, $M$ has no $U_{2, 5}$-minor and (\ref{sublem: rk3 no U25}) holds.

\begin{sublemma}
	\label{sublem: rk3 no F7}
	$M$ has no $F_7$-minor or $F_7^*$-minor.
\end{sublemma}

Since $r(M)=3$, it follows that $M$ has no $F_7^*$-minor. Thus it remains to show that $M$ has no $F_7$-minor. As $|E(M)|=8$, this reduces to checking that $M$ is not a single-element extension of $F_7$. It is easily checked that each single-element extension of $F_7$ has a loop, a non-trivial parallel class, a $U_{2, 5}$-restriction, or a $U_{3, 5}$-restriction. Since $M$ is $3$-connected, it follows by (\ref{sublem: rk3 no U25}) that (\ref{sublem: rk3 no F7}) holds.

Thus, $M$ is ternary. We now note that every simple rank-$3$ ternary matroid with no $U_{2, 4}$-restriction is a restriction of $AG(2, 3)$. Since $|E(M)|=8$ and $M$ is $3$-connected, we deduce that $M\cong  AG(2, 3)\backslash e$. This completes the proof of Lemma~\ref{most3}.
\end{proof}

\begin{lemma}
Let $M$ be a $3$-connected matroid with rank and corank at least four. Then $M$ is minimally not in $\cS_3$ if and only if $M$ is isomorphic to one of $P_8$, $M(\cW_4)$, or $\cW^4$.
\label{least4}
\end{lemma}

\begin{proof}
Using Theorem~\ref{main1}, it is easily checked that each of the $3$-connected matroids $P_8$, $M(\cW_4)$, and $\cW^4$ is minimally not in $\cS_3$. Now suppose that $M$ is a $3$-connected matroid that is minimally not in $\cS_3$. Let
$$E_d=\{e\in E(M) : M\backslash e \text{ is $3$-connected} \}$$
and
$$E_c=\{e\in E(M) : M/ e \text{ is $3$-connected} \}.$$
If $E_c$ and $E_d$ are both empty, then, by Tutte's Wheels-and-Whirls Theorem~\cite{tut66}, $M$ is isomorphic to either a wheel or a whirl. In particular, $M$ is isomorphic to either $M(\cW_4)$ or $\cW^4$. Therefore, assume that at least one of $E_c$ or $E_d$ is non-empty.

\begin{sublemma}
\label{sublem: no 4pl}
$M$ has no $U_{2, 4}$-restriction.
\end{sublemma}

Suppose there exists $X \subseteq E(M)$ such that $M|X\cong U_{2, 4}$. By Lemma~\ref{segdelete}, $M\backslash e$ is $3$-connected for all $e\in X$, and so $M\backslash e$ is a spike minor. For all $e \in X$, the set $X - \{e\}$ is a triangle of $M \backslash e$. Therefore, as $r(M) \geq 4$, there is a unique element of $X-\{e\}$ which is the tip of $M\backslash e$. In particular, at least two elements of $X$, say $f$ and $g$, are each contained in a triangle of $M$ not in $X$. But then, by Lemma~\ref{lemma: spike degeneracy}, $M\backslash h$, where $h\in X-\{f, g\}$, is not a spike minor, a contradiction. Thus, (\ref{sublem: no 4pl}) holds.

\begin{sublemma}
\label{sublem: T,T*}
If $M$ has a triangle $T$, then $E_d\subseteq T$. Dually, if $M$ has a triad $T^*$, then $E_c\subseteq T^*$.
\end{sublemma}

Suppose that $T=\{t, a, b\}$ is a triangle of $M$ and suppose to the contrary that there is an element $x\in E_d-T$. Consider $M\backslash x$, and let $r$ denote the rank of $M$. As $M\backslash x$ is $3$-connected, it follows by Lemma~\ref{3conspike} that $M\backslash x$ is isomorphic to either $\Phi_r$ or $\Phi_r\backslash z$, where $\Phi_r$ is a rank-$r$ tipped spike and $z$ is a non-tip element in $E(\Phi_r)$. Without loss of generality, we may assume that $t$ is the tip of $M\backslash x$. First assume that $M\backslash x\cong \Phi_r$ and let $e\in E(M)-\{t, x\}$. Then, by Lemma~\ref{3conspike}, $(M\backslash x)\backslash e\cong \Phi_r\backslash e$ is $3$-connected. Since $M$ is $3$-connected, $M\backslash e$ is also $3$-connected. In particular, $M\backslash e$ is a tipped spike whose tip element is $t$. Thus $\{x, f\}$ is a leg of $M\backslash e$ for some element $f\in E(M)-\{t, e, x\}$. The set $\{x,f,t\}$ is a triangle of $M \backslash e$, and $f$ is not contained in any other triangle of $M \backslash e$. In $M \backslash x$, the element $f$ is contained in a triangle with the tip $t$. This triangle must be $\{e,f,t\}$. But then $\{e, f, x, t\}$ is a $U_{2, 4}$-restriction in $M$, contradicting (\ref{sublem: no 4pl}).

Now assume that $M\backslash x\cong \Phi_r\backslash z$ for some non-tip element $z\in E(\Phi_r)$. Let $g$ be the unique element in a half leg of $M \backslash x$. Then $(M/g)\backslash x=(M\backslash x)/ g$ is a rank-($r-1$) tipped spike, and so $(M/g)\backslash x$ is $3$-connected. However, $M/g$ is not $3$-connected as it has one too many elements to be a simple restriction of a rank-$(r-1)$ tipped spike. Thus, $x$ is in a triangle with $g$ in $M$. If $\{t, g, x\}$ is a triangle, then $M$ is isomorphic to a tipped spike, a contradiction. Therefore, there is an element $e\in (E(M)-\{t, g, x\})$ such that $\{e, g, x\}$ is a triangle of $M$.

By Lemma~\ref{3conspike}, the matroid $M \backslash x \backslash t$ is $3$-connected, so $M \backslash t$ is also $3$-connected. Hence, $M \backslash t$ is a $3$-connected restriction of a tipped spike. Since $M \backslash t$ contains the triangle $\{e,g,x\}$, Lemma~\ref{3conspike} implies that $M \backslash t$ is isomorphic to a tipped spike with a single non-tip element deleted. The tip of $M \backslash t$ must be $x$, as otherwise $M \backslash x$ would contain a triangle which does not contain the tip $t$.

Let $\{a,b\}$ be a full leg of $M \backslash x$ which does not contain $e$. The set $\{t,a,b\}$ is a triangle of $M$, so Tutte's Triangle Lemma implies that either $M \backslash a$ or $M \backslash b$ is $3$-connected if $\{t,a,b\}$ does not intersect a triad of $M$. The only triad of $M \backslash x$ which intersects $\{t,a,b\}$ is $\{a,b,g\}$. By orthogonality with the triangle $\{e,g,x\}$, the set $\{a,b,g\}$ is not a triad of $M$. Therefore, $\{t,a,b\}$ does not intersect a triad of $M$, so, without loss of generality, we may assume that $M \backslash a$ is $3$-connected. However, since $r(M) \geq 4$, each of $M \backslash x$ and $M \backslash t$ contain two triangles which do not contain the element $a$. Therefore, $M \backslash a$ has two triangles which intersect at $t$ and two triangles which intersect at $x$. This means that $M \backslash a$ is not a spike minor, a contradiction which completes the proof of (\ref{sublem: T,T*}).

\begin{sublemma}
\label{sublem: xy}
If $x\in E_d$ and $y\in E(M)-(E_c\cup \{x\})$, then $\{x, y\}$ is a subset of a triangle. Dually, if $x\in E_c$ and $y\in E(M)-(E_d\cup \{x\})$, then $\{x, y\}$ is a subset of a triad.
\end{sublemma}

Let $x\in E_d$ and $y\in (E(M)-(E_c \cup \{x\}))$. If $M\backslash x$ has a triangle, then we have a contradiction to (\ref{sublem: T,T*}). Thus, $M\backslash x$ is isomorphic to either $\Phi_r\backslash t$ or $\Phi_r\backslash \{t, z\}$, where $\Phi_r$ is a rank-$r$ tipped spike with tip $t$ and $z\in E(\Phi_r)-\{t\}$. In turn, this implies that $M\backslash x/y$ is $3$-connected. As $M/y$ is not $3$-connected, we deduce that $\{x, y\}$ is a subset of a triangle, thereby proving (\ref{sublem: xy}).

\begin{sublemma}
\label{sublem: EdEc ne}
$E_d\neq \emptyset$ and $E_c\neq \emptyset$.
\end{sublemma}

Up to duality, $E_d$ is non-empty. Let $t\in E_d$. If $E_c$ is empty, then, by (\ref{sublem: T,T*}) and (\ref{sublem: xy}), $t$ is the unique element in $E_d$ and every other element is in a triangle with $t$. By (\ref{sublem: no 4pl}), each of these triangles intersect only at $t$. Since the matroid $M \backslash t$ is a spike minor, we see that the disjoint union of any two triangles is a cocircuit. Therefore, $M$ is a restriction of a tipped spike with tip $t$, a contradiction. Hence, $E_c$ is non-empty, thus proving (\ref{sublem: EdEc ne}).

\begin{sublemma}
\label{sublem:notriangle}
$M$ has no triangles and no triads.
\end{sublemma}

Suppose that $M$ has a triangle $T=\{t, a, b\}$. By (\ref{sublem: T,T*}), $E_d\subseteq T$. Since $E_d$ is non-empty, we may assume that $t\in E_d$.

First assume that $M$ has a triad $T^*$. By (\ref{sublem: T,T*}) and (\ref{sublem: EdEc ne}), $E_c\subseteq T^*$ and $E_c$ is non-empty. Let $t^*\in E_c$. If $e$ is an element of $E_c-\{t^*\}$, then, by (\ref{sublem: T,T*}), $e\in T^*$. Let $f \in E(M) - (T \cup T^*)$. By (\ref{sublem: T,T*}), $f \notin E_d$ and $f \notin E_c$. Therefore, by (\ref{sublem: xy}), there is a triad containing $\{t^*, f\}$. By (\ref{sublem: T,T*}), this triad contains $\{e\}$, so $\{t^*, e, f\}$ is a triad of $M$. Furthermore, by (\ref{sublem: T,T*}) again, $\{t, f\}$ is contained in a triangle. But neither $t^*$ nor $e$ are in this triangle (as $M / t^*$ and $M / e$ are $3$-connected), a contradiction to orthogonality. Thus, $E_c=\{t^*\}$. By a dual argument, $E_d=\{t\}$. By (\ref{sublem: xy}), for all $f\in E(M)-\{t, t^*\}$, we have that $f$ is in a triangle with $t$ and $f$ is in a triad with $t^*$. In particular, to avoid a contradiction to orthogonality, there is a partition $\{\{x_1, x_2, \ldots, x_k\}, \{y_1, y_2, \ldots, y_k\}\}$ of $E(M)-\{t, t^*\}$ such that $\{x_i, y_i, t\}$ is a triangle and $\{x_i, y_i, t^*\}$ is a triad for all $i$. But then $M$ is a spike with tip and cotip by Lemma~\ref{lem: tip_and_cotip}, a contradiction. Thus, $M$ has no triads.

By (\ref{sublem: EdEc ne}), $E_c$ is non-empty. Let $x\in E_c$. Then, as $|E_d|\le 3$ by (\ref{sublem: T,T*}) and $|E(M)|\ge 8$, there exists an element $y\in (E(M)-(E_d\cup \{x\}))$. But then, by (\ref{sublem: xy}), $\{x, y\}$ is contained in a triad, a contradiction. Therefore $M$ has no triangles and, by duality, $M$ has no triads. Thus, (\ref{sublem:notriangle}) holds.

\begin{sublemma}
\label{sublem: Ed Ec}
If $x\in E_d$, then $(E(M)-\{x\})\subseteq E_c$. Dually, if $y\in E_c$, then $(E(M)-\{y\})\subseteq E_d$.
\end{sublemma}

Let $x\in E_d$. By (\ref{sublem:notriangle}), $M\backslash x$ is a $3$-connected restriction of a tipless spike. If $y\in E(M)-\{x\}$, then $M\backslash x/y$ is a $3$-connected restriction of a tipped spike, and so $M/y$ is $3$-connected unless $x$ and $y$ are contained in a triangle. By (\ref{sublem:notriangle}), $M$ has no triangles, and so (\ref{sublem: Ed Ec}) holds.

By repeated applications of (\ref{sublem: Ed Ec}), we deduce that $E_d=E_c=E(M)$. In particular, every single-element deletion and every single-element contraction of $M$ is a spike minor, and so $M$ is an excluded minor for $\cS$ of rank and corank at least four. In particular, by Theorem~\ref{main1}, $M\cong P_8$. This completes the proof of Lemma~\ref{least4}.
\end{proof}

\begin{proof}[Proof of Theorem~\ref{them: 3con}]
Since $\cS_3$ is closed under duality, it follows by Lemma~\ref{most3} that the theorem holds if $M$ has rank or corank at most three. Theorem~\ref{them: 3con} now follows by Lemma~\ref{least4}.
\end{proof}

\section*{Acknowledgements} The second and sixth authors were supported by the New Zealand Marsden Fund. The work for this paper began at the 2nd New Zealand Postgraduate Workshop in Matroids held in Westport, New Zealand in 2021.

\end{document}